\documentclass[11pt]{amsart}
\usepackage[margin=1.5in]{geometry}

\usepackage{mathrsfs}
\usepackage{amsfonts}
\usepackage{latexsym,epsfig}
\usepackage{mathabx}
\usepackage{amsmath, amscd, amsthm,amssymb}
\usepackage[pdftex]{color}
\usepackage{comment}
\usepackage{marginnote}
\usepackage[colorlinks,citecolor=blue,linkcolor=red]{hyperref}
\usepackage[nameinlink]{cleveref}
\usepackage{enumitem}
\usepackage{enumitem}
\usepackage{pinlabel}
\usepackage[dvipsnames]{xcolor}

\newcommand{\Z}{\mathbb{Z}}

\newcommand{\R}{\mathbb{R}}

\newcommand{\wt}{\widetilde}
\newcommand{\mc}{\mathcal}
\newcommand{\mf}{\mathfrak}

\newcommand{\ol}{\overline}
\newcommand{\wh}{\widehat}

\newcommand{\mr}{\mathring}

\newcommand{\cut}{\backslash \!\! \backslash}

\usepackage{todonotes}

\usepackage{hyperref}
\usepackage{pinlabel}

\newtheorem{theorem}{Theorem}[section]
\newtheorem{lemma}[theorem]{Lemma}
\newtheorem{proposition}[theorem]{Proposition}
\newtheorem{corollary}[theorem]{Corollary}

\newtheorem{claim}{Claim}

\theoremstyle{definition}

\newtheorem{remark}[theorem]{Remark}

\newcommand{\FF}{\mathcal{F}}

\newcommand{\orb}{\mathcal{O}}

\newcommand\tsim{\kern-.4em\sim}

\newcommand\ssm{\smallsetminus}

\renewcommand{\phi}{\varphi}
\renewcommand{\epsilon}{\varepsilon}

\newcommand{\pp}{\mathcal{P}}

\title[Depth one laminations transverse to pseudo-Anosov flows]{Constructing depth one laminations transverse to pseudo-Anosov flows}

\author{Junzhi Huang and Samuel J. Taylor}

\begin{document}

\begin{abstract}
Given a pseudo-Anosov flow $\phi$ on a closed atoroidal $3$--manifold $M$ and a closed surface $S$ almost transverse to $\phi$, we give a homological characterization of when $S$ can be completed to an almost transverse depth one lamination or foliation whose set of compact leaves is $S$. As a consequence, we show that the cone of classes in $H^1(M\cut S)$ that are positive on the closed orbits of $\phi$, when nonempty, is an entire foliation cone of $M\cut S$.
\end{abstract}

\maketitle

\setcounter{tocdepth}{1}
\tableofcontents

\section{Introduction}
\label{sec:intro}

Given a pseudo-Anosov flow $\varphi$ on a closed atoroidal $3$--manifold $M$ and an embedded transverse surface $S$, it is easy to see that $S$ can be completed to a transverse fibration of $M$ if and only if $S$ is a cross section of $\varphi$, i.e. $S$ is transverse to and meets every orbit of $\varphi$. 
Indeed, a cross section $S$ admits a well-defined first return map under $\varphi$ and so in this case $\phi$ is naturally orbit equivalent to the suspension flow of its first return map.
However, even when $\varphi$ is a suspension flow (that is, $\phi$ admits \emph{some} cross section), a given transverse surface $S$ may fail to intersect some orbits of $\varphi$.
Such phenomenon occurs precisely when the second homology class of $S$ lies on the boundary of the fiber cone associated to $\phi$. In this case, one can still ``spin" a cross section of $\phi$
around $S$ and produce a transverse depth one foliation of $M$ with $S$ being the unique closed leaf. (This idea goes back to Thurston \cite{thurston1986norm}, but an explicit description can be found, for example, in \cite[Section 2.3]{landry2023endperiodic}.) Hence, for a suspension pseudo-Anosov flow every transverse surface can be completed to either a transverse depth zero foliation (i.e. a fibration) or depth one foliation.

\smallskip
Since a general pseudo-Anosov flow $\varphi$ does not admit a cross section, it is then natural to ask when a given surface $S$ transverse to $\varphi$ can be completed to a transverse foliation, or more generally, a transverse lamination. Here, we consider the question also for surfaces that are almost transverse to $\varphi$.
 
 \smallskip

Our main theorem provides a homological characterization of when such an extension is possible, if we require the completed lamination or foliation to be of \emph{depth one}. This means that not all leaves are compact and for any non-compact leaf $L$, $\ol{L} \ssm L$ consists only of compact leaves. Recall that a closed surface $S$ is said to be \emph{almost transverse} to a pseudo-Anosov flow $\phi$ if there is a dynamic blowup of $\phi$ that is transverse to $S$ (see \Cref{sec:pA}). Let $\phi^\#$ be a blownup transverse flow.
We say a class $\eta \in H^1(M\cut S)$ is \emph{nonnegative} or \emph{positive} if it is nonnegative or positive on each oriented closed orbit of $\phi^\sharp$ that is contained in $M\cut S$. This definition does not depend on the choice of the transverse blowup; see \Cref{sec:nonneg-class}. 
Here, $M \cut S$ is the compact manifold with boundary obtained by cutting $M$ along $S$.

For a lamination or foliation $\mc L$, let $\mc L^0$ denote its compact leaves.
\begin{theorem}\label{thm:depth-one}
Let $\varphi$ be a pseudo-Anosov flow on a closed atoroidal $3$-manifold $M$ and let $S$ be a closed almost transverse surface. Then $S$ can be completed to an almost transverse depth one lamination (resp. foliation) $\mc L$ with $\mc L^0 = S$ if and only if there exists a nonnegative (resp. positive) class in $H^1(M \cut S)$. 
\end{theorem}

\Cref{thm:depth-one} will follow from a relative version of the transverse surface theorem. The classical transverse surface theorems, established by Mosher \cite{Mos92,mosher1992dynamical}, Landry \cite{Landry_norm}, and Landry--Minsky--Taylor \cite{landry2025transverse}, connects transversality of closed surfaces in 3-manifolds, possibly with toral boundary, with cohomological conditions. We extend these results to 3-manifolds with boundary and to properly embedded compact surfaces with respect to a `cut' pseudo-Anosov flow on the manifold $M \cut S$. Recall that a properly embedded surface is \emph{essential} if it is both incompressible and boundary incompressible. A properly embedded surface $\Sigma$ in $M \cut S$ is transverse to a flow $\phi$ in $M$ if $\Sigma$ is transverse to $\phi$ in its interior, and there is a way to smooth $\partial\Sigma$ on $S$ so that the branched surface $S\cup\Sigma$ is transverse.

\begin{theorem}\label{thm:transverse-surface}
    Let $M$ be a closed atoroidal 3-manifold with a pseudo-Anosov flow $\phi$. Let $S$ be a closed surface in $M$ that is almost transverse to $\phi$. For any nonnegative class $\eta\in H^1(M\cut S)$, there exists a properly embedded essential surface $\Sigma\subseteq M\cut S$ representing $\eta$ and a dynamic blowup of $\phi$ that is transverse to both $S$ and $\Sigma$.
\end{theorem}

We say a depth one lamination or foliation \emph{represents} a class $\eta\in H^1(M\cut S)$ if for any loop $\gamma$ in $M \cut S$, the value $\eta(\gamma)$ is
the algebraic intersection number of $\gamma$ with a depth one leaf of the lamination. 
By spinning $\Sigma$ in the statement of \Cref{thm:transverse-surface}, we obtain the following characterization of classes represented by transverse laminations or foliations, implying \Cref{thm:depth-one}.

\begin{theorem}\label{thm:class}
Let $\varphi$ be a pseudo-Anosov flow on a closed atoroidal $3$-manifold $M$ and let $S$ be a closed almost transverse surface. Then a class $\eta$ in $H^1(M\cut S)$ is represented by a depth one almost transverse lamination (resp. foliation) $\mc L$ with $\mc L^0 = \partial(M \cut S)$ if and only if $\eta$ is nonnegative (resp. positive).
\end{theorem}

\begin{remark}[Connections to the literature]
Using the Schwartzman--Sullivan theory of foliation currents, Cantwell and Conlon prove that if $X$ is a smooth oriented $1$-dimensional foliation of $N = M \cut S$ and $\eta \in H^1(N)$ is strictly positive on the asymptotic cycles of $X$, which are a certain generalization of closed orbits, then $\eta$ can be represented by a smooth foliated form $\omega$ that is transverse to $X$ on $\rm{int}(N)$ \cite[Theorem 4.9]{CC2013}. 
(Technically, their Section 4.3 starts by assuming that $X$ is already transverse to \emph{some} foliation $\mc F$ without holonomy, but this assumption is not used in the proof.)
When $\eta$ is an integral class, the foliation tangent to the kernel of $\omega$ is depth one and transverse to $X$.

This result gives another potential path towards proving \Cref{thm:depth-one} in the special case where the form $\eta$ is positive. However, this approach has its own technical difficulties; for example, in the setting of \Cref{thm:depth-one} the $1$-dimensional foliation of $N$ induced by the flow $\varphi$ is not smooth since $\varphi$ is itself not smooth on $M$. Our approach is combinatorial, using the veering triangulation associated to $\varphi$, and so low regularity is not an issue. Moreover, the strictly positive case is an easy consequence of the more general statement about nonnegative classes and depth one laminations, which itself does not follow from the Cantwell--Conlon approach.
\end{remark}

\subsection*{Applications}
\subsubsection*{Core complexity}
\Cref{thm:transverse-surface} gives a homological condition for the manifold $M\cut S$ to have finite \emph{core complexity} $\mf c(M \cut S)$, defined in \cite{HT2026}.  In more detail, when $S$ is a closed surface almost transverse to a pseudo-Anosov flow $\varphi$, the core complexity is defined to be

\begin{align*}
\mathfrak{c}(M\cut S) = 
\begin{cases}
0, & M\cut S \text{ contains a product annulus},\\
\min{-\chi(\Sigma)}, & \text{otherwise},
\end{cases}
\end{align*}
where a product annulus is an essential annulus joining distinct boundary components of $M\cut S$. In the second case, the minimum is over all properly embedded surfaces $\Sigma$ in $M\cut S$ that are transverse to (a dynamic blowup of) $\varphi$ and which also meet both boundary components of $M \cut S$.

The main theorems in \cite{HT2026} control the hyperbolic geometry of $M$ in terms of curve graph data in the surface $S$ associated to the flow $\varphi$. These theorems however apply only when $\mathfrak{c}(M \cut S)$ is finite. The following corollary of \Cref{thm:transverse-surface} states that this finiteness is really a cohomological condition.

\begin{corollary}
If there is a nonnegative class in $H^1(M\cut S)$ that is nontrivial when restricted to each component of $\partial(M\cut S)$, then $\mf c(M \cut S) < \infty$.
\end{corollary}

\subsubsection*{Foliation cones}
Each depth one foliation of $M$ that contains $S$ as its collection of depth zero (i.e. closed) leaves determines a depth one foliation of $N  = M \cut S$  whose (cooriented) closed leaves are exactly $\partial_\pm N$. Here, $\partial_\pm N$ are the components of $\partial N$ that are cooriented out of or in to $N$, respectively. 

In general, $N$ admits many such depth one foliations and these are organized by the \emph{foliation cones} of $H^1(N)$:
A class in $H^1(N)$ is said to be \emph{foliated} if it is dual to a fibration $N \ssm \partial N \to S^1$ whose foliation by fibers extends to $N$ by attaching $\partial_\pm N$.
Cantwell and Conlon prove that the foliated classes in $H^1(N)$ are precisely
 the integer points of
a union of finitely many open rational polyhedral cones, each of which is called a \emph{foliation cone} of $N$ \cite{cantwell1999foliation, CCF19, landry2023endperiodic}. 

The next result, which follows from the proof of \Cref{thm:class}, states that any pseudo-Anosov flow on $M$ that is almost transverse to $S$ `sees' an entire foliation cone of $M \cut S$. The depth one foliations of this cone are exactly the ones represented by positive classes of $H^1(M\cut S)$ and these foliations can each be realized transverse to the induced semiflow on $M \cut S$.

\begin{theorem} \label{th:intro_cone}
Let $\varphi$ be a pseudo-Anosov flow on a closed atoroidal $3$-manifold $M$ and let $S$ be a closed almost transverse surface. The cone in $H^1(M\cut S)$
consisting of positive classes, when nonempty, is a foliation cone of $M\cut S$.
\end{theorem}

We additionally prove that the cone of positive classes is nonempty if and only if there is a nonnegative class that is contained in a foliation cone of $M\cut S$ (see \Cref{prop:nonempt}).

\subsection*{Acknowledgements}
We thank Chi Cheuk Tsang for helpful comments on an earlier draft of this paper. The completion of this paper was supported by the National Science Foundation under Grant No. DMS--2424139, while the authors were in residence at the Simons Laufer Mathematical Sciences Institute in Berkeley, California, during the Spring 2026 semester.
Taylor was also partially supported by NSF grant DMS--2503113 and the Simons Foundation, and Huang was partially supported by NSF grant DMS-2005328.

\section{Background}
Here, we briefly review the background needed for the paper and provide references where further details can be found.

\subsection{(Almost) pseudo-Anosov flows and their orbit spaces}\label{sec:pA}

We refer the readers to the recent monograph by Barthelm\'e and Mann \cite{barthelme2025pseudo} for a detailed treatment of the theory. Roughly speaking, a flow $\phi$ on a closed 3-manifold $M$ is a \emph{pseudo-Anosov} flow if it preserves a pair of transverse 2-dimensional singular foliations $\FF^s$ and $\FF^u$ on $M$, called the stable and the unstable foliations, so that every leaf of $\FF^s$ (resp. $\FF^u$) is a union of flow lines that are forward (resp. backward) asymptotic. The singular locus of each foliation is a finite collection of closed \emph{singular orbits}, 
and dynamics near a singular orbit is locally modeled on a pseudo-hyperbolic orbit with alternating contracting directions (locally determining $\mc F^s$) and repelling directions (locally determining $\mc F^u$).

\smallskip
We also consider \emph{almost pseudo-Anosov flows}, which are a generalization of pseudo-Anosov flows first defined by Mosher \cite{mosher1990correction, mosher1992dynamical}, and are obtained by dynamically blowing up a collection of singular orbits of pseudo-Anosov flows. See also Landry--Minsky--Taylor \cite{landry2025transverse} for an explicit description of the construction.

To blow up a singular orbit of a pseudo-Anosov flow $\phi$, we replace the singular orbit by a \emph{blowup complex}, which is a union of annuli, called \emph{blowup annuli}, glued along their boundaries. Each blowup annulus is a union of flowlines, with two closed orbits at the boundary and with orbits in the interior converging to distinct boundaries in forward and backward time. Moreover, the flow restricted to each blowup annulus has a cross section (so not a Reeb annulus). The blowup complex has a transverse section which is a finite tree. Each interior vertex of the tree has even degree, and the first return map induced by the blown up flow near an interior vertex is locally alternating between contracting and repelling on adjacent edges. See \Cref{fig:blowup} for an example.

\begin{figure}[h!]
    \centering
    \includegraphics[width= .5\textwidth]{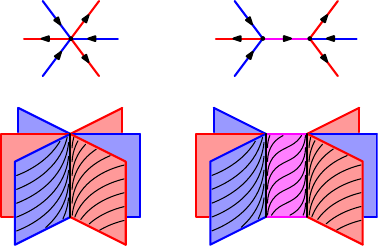}
    \caption{Figure from \cite{landry2023endperiodic} showing a dynamic blowup of a 3-pronged singular orbit and its transverse cross section. There is a single blowup annulus in this example.}\label{fig:blowup}
\end{figure}

An almost pseudo-Anosov flow also has a pair of stable and unstable foliations. The two foliations are transverse except along the blowup annuli, and these annuli are contained in leaves of each foliation. If an almost pseudo-Anosov flow $\phi^\sharp$ is obtained as a blowup of a pseudo-Anosov flow $\phi$, then there is a blowdown map $M \to M$, which is homotopic to the identity, that collapses each blowup complex of $\phi^\sharp$ to a single singular orbit and also semiconjugates the flow $\varphi^\sharp$ with its invariant foliations to $\varphi$ with its invariant foliations.

\medskip
In general, given a flow $\psi$ on a manifold $N$, the \emph{orbit space} (also called the \emph{flow space}) of $\psi$ is the quotient of the universal cover $\wt{N}$ of $N$ by the lifted flowlines.

Let $\phi$ be an almost pseudo-Anosov flow with stable and unstable foliations $\mc F^{s/u}$, and let $\orb$ be the orbit space of $\phi$.
By Fenley and Mosher \cite{fenley2001quasigeodesic}, $\orb$ is homeomorphic to $\R^2$ . Denote the natural projection map by $\Theta\colon \wt M \to \orb$. The stable and unstable foliations $\FF^s$ and $\FF^u$ lift to singular foliations $\wt{\FF}^s$ and $\wt{\FF}^u$ of $\wt{M}$, and they descend to a pair of transverse 1-dimensional singular foliations on $\orb$. We denote them by $\orb^s$ and $\orb^u$ respectively. The $\pi_1(M)$--action on $\wt{M}$ induces a $\pi_1(M)$-action on $\orb$ by homeomorphisms preserving $\orb^s$ and $\orb^u$. Define a \emph{leaf slice} of a leaf $\ell$ in $\orb^s$ or $\orb^u$ to be a properly embedded copy of $\R$ in $\ell$. If $\ell$ is a periodic leaf, i.e. contains a periodic point for the $\pi_1(M)$--action, then a \emph{half-leaf} of $\ell$ is a complementary component of the periodic points in $\ell$. Note that a half-leaf can either be an infinite ray or a \emph{blowup segment}, which is the projection of some lift of a blowup annulus.

A \emph{perfect fit rectangle} in $\orb$ is a proper embedding of a rectangle with a corner missing, i.e. $[0,1]\times[0,1]\ssm\{(1,1)\}$, into $\orb$ so that the vertical/horizontal foliation maps to $\orb^{s/u}$. We say a subset $A$ in $M$ \emph{kills perfect fits} (or that $\varphi$ \emph{has no perfect fits relative to} $A$) if every perfect fit rectangle in $\orb$ contains a point in the shadow $\Theta(\wt A)$ of the preimage $\wt A$ of $A$ in $\wt{M}$. For any transitive pseudo-Anosov flow $\phi$ on a closed $3$-manifold $M$, Tsang proves that there is a finite collection of closed orbits of $\phi$ that kills its perfect fits \cite[Proposition 2.7]{tsang2022constructing}. This fact also follows from earlier work of Fried \cite{fried1983transitive} and Brunella \cite{brunella1995surfaces} which proves the existence of Birkhoff sections for such flows.

\smallskip

\subsection{Veering triangulations}\label{sec:veering}

Our main tool to construct transverse laminations is veering triangulations. These are combinatorial objects that encode the dynamics of (almost) pseudo-Anosov flows. We sketch below the construction due to Agol-Gu\'eritaud, and we also refer the readers to \cite{LMT21} for details. 

First suppose $\phi$ is a pseudo-Anosov flow. The case for an almost pseudo-Anosov flows is very similar, and is discussed in \Cref{sec:apA}. Let $\kappa$ be the union of all singular orbits and a finite collection of regular closed orbits so that $\kappa$ kills perfect fits, and set $M_0 = M \ssm \kappa$.
If we denote the universal cover of $M_0$ by $\wt{M_0}$, then the orbit space $\mr\pp$
 for $\phi|_{M_0}$
 is by definition the quotient of $\wt{M_0}$ by the lifted flow lines. The orbit space $\mr\pp$ inherits a pair of transverse (non-singular) 1-dimensional foliations $\mr\pp^{s/u}$.

There is a natural way to define a \emph{completed orbit space} $\pp$ by adding ideal points at infinity (see the discussion in \cite[Section 4]{LMT21}). First note that $\mr\pp$ is naturally identified with the universal cover of $\orb\ssm\wh{\kappa}$, where $\wh{\kappa} = \Theta(\wt \kappa)$ is projection to $\orb$ of the preimage $\wt \kappa$ of $\kappa$ in $\wt M$.
 Then $\pp$ is defined to be the corresponding infinite branched covering of $\orb$ branching over $\wh{\kappa}$, and the ideal (i.e. completion) points are the points of the preimage of $\wh{\kappa}$.  
The completed flow space $\pp$ also inherits a pair of transverse foliations $\pp^{s/u}$ with endpoints of the leaves possibly at the ideal points.

A \emph{rectangle} in $\pp$ is an embedding of $[0,1]\times[0,1]$ into $\pp$ so that on its interior, it sends horizontal lines to leaf segments of $\pp^u$ and vertical lines to leaf segments of $ \pp^s$. A rectangle is called a \emph{maximal rectangle} (also called a \emph{tetrahedra rectangle}) if it contains an ideal point in the interior of each of its sides. The condition that $\kappa$ kills perfect fits implies that every increasing union of rectangles is contained in a maximal rectangle.
Note that a maximal rectangle projects to a rectangle in $\orb$ with a point of $\wh{\kappa}$ in the interior of each side.
\smallskip

A \emph{taut ideal tetrahedron}, defined by Lackenby \cite{lackenby2000taut}, is an ideal tetrahedron with a coorientation for its faces so that it has two faces with the coorientation pointing out, called the \emph{top faces}, and two faces with the coorientation pointing in, called the \emph{bottom faces}. Each edge of a taut ideal tetrahedron is then assigned an angle of $0$ or $\pi$ according to this coorientation. The $\pi$-edge between two top faces is the \emph{top edge}, and the $\pi$-edge between two bottom faces is the \emph{bottom edge}. See \Cref{fig:taut-tetrahedron} for a picture of a taut tetrahedron with the ends truncated (see also \Cref{sec:ladderpole}).

\begin{figure}[h]
\begin{center}
\includegraphics[width = .4 \textwidth]{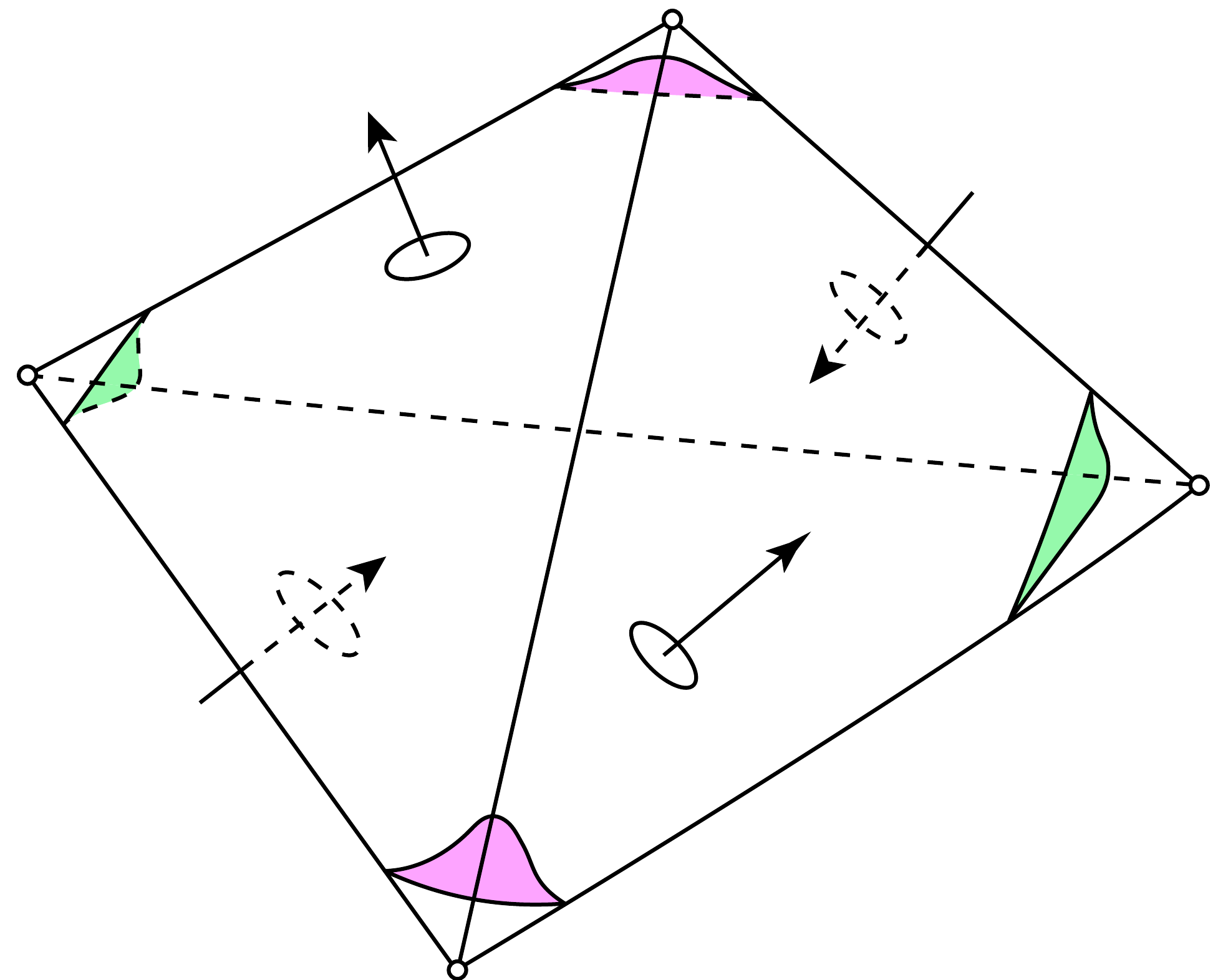}
\caption{An illustration of a taut ideal tetrahedron with coorientations of the faces. Tips of the truncated tetrahedron are shown in color, with two upward tips in purple and two downward tips in green.}
\label{fig:taut-tetrahedron}
\end{center}
\end{figure}

As shown in \Cref{fig:2-tets}, we associate to each maximal rectangle in $\pp$ a taut ideal tetrahedron with the convention that the edge connecting two horizontal sides is the top edge. Two ideal tetrahedra are glued along a face if their maximal rectangles share three ideal points on their boundaries.
This defines a space with an ideal triangulation $\wt{\tau}$. Since the action of $\pi_1(M_0)$ on $\pp$ preserves maximal rectangles, it induces a covering simplicial action on the geometric realization of $\wt{\tau}$, and the quotient space is triangulated and homeomorphic to $M_0$. Therefore, this construction produces an ideal triangulation $\tau$ on $M_0$, called the associated \emph{veering triangulation}.

\begin{figure}[h]
\begin{center}
\includegraphics[width = .4 \textwidth]{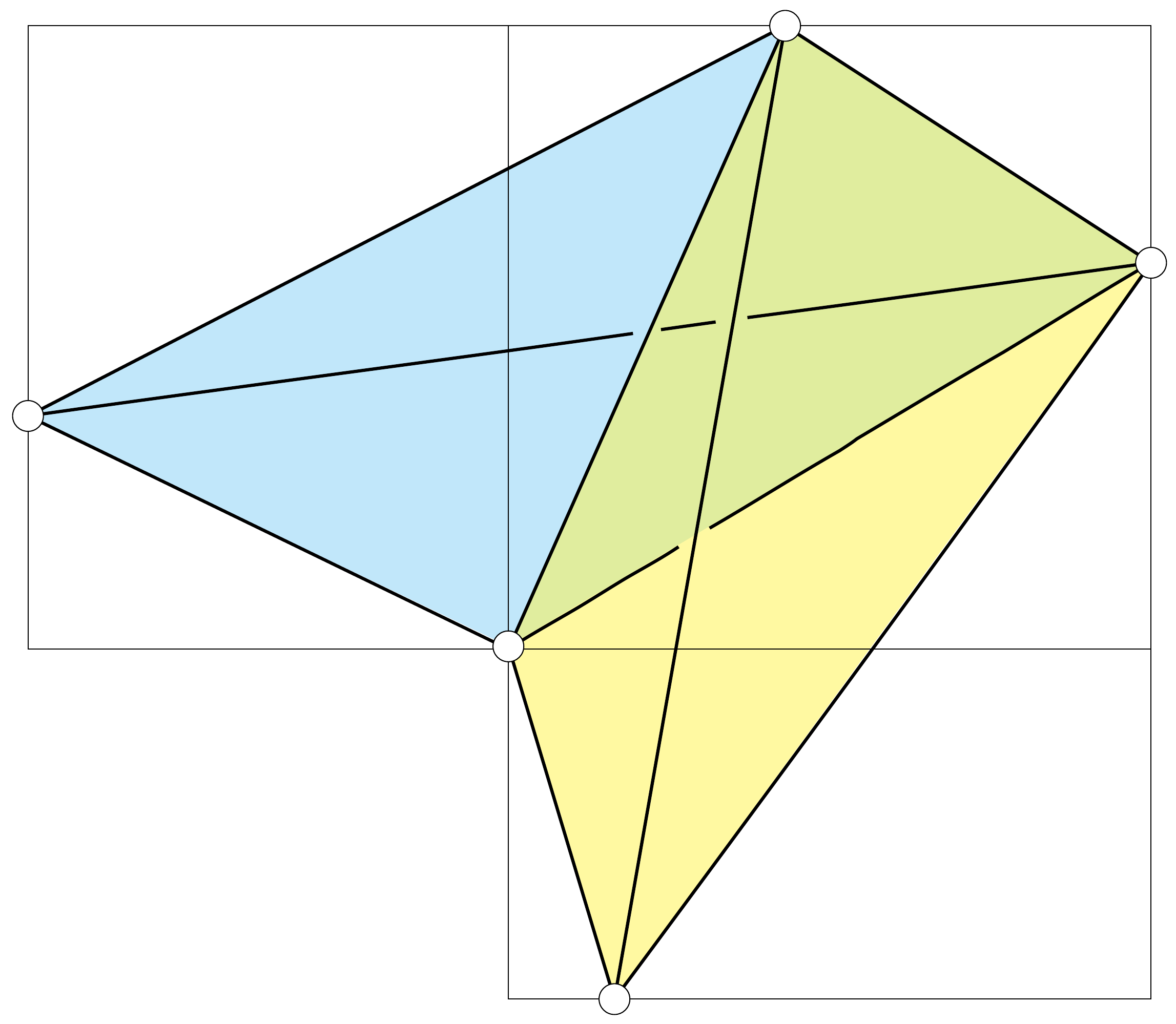}
\caption{Two tetrahedra rectangles sharing three ideal points on their boundaries give rise to a gluing between their tetrahedra along a face.}
\label{fig:2-tets}
\end{center}
\end{figure}

The triangulation $\tau$ is \emph{taut} in the sense of \cite{lackenby2000taut}, which means that around every edge the angles sum up to $2\pi$. Hence, the 2-skeleton $\tau^{(2)}$ is naturally a cooriented branched surface, branching along the edges. See \Cref{fig:branch-edge} for a local picture near an edge. By \cite[Theorem C]{LMT21}, after an isotopy, the oriented branched surface $\tau^{(2)}$ is positively transverse to the orbits of the flow $\phi$.

\begin{figure}[h]
\begin{center}
\includegraphics[width = .4 \textwidth]{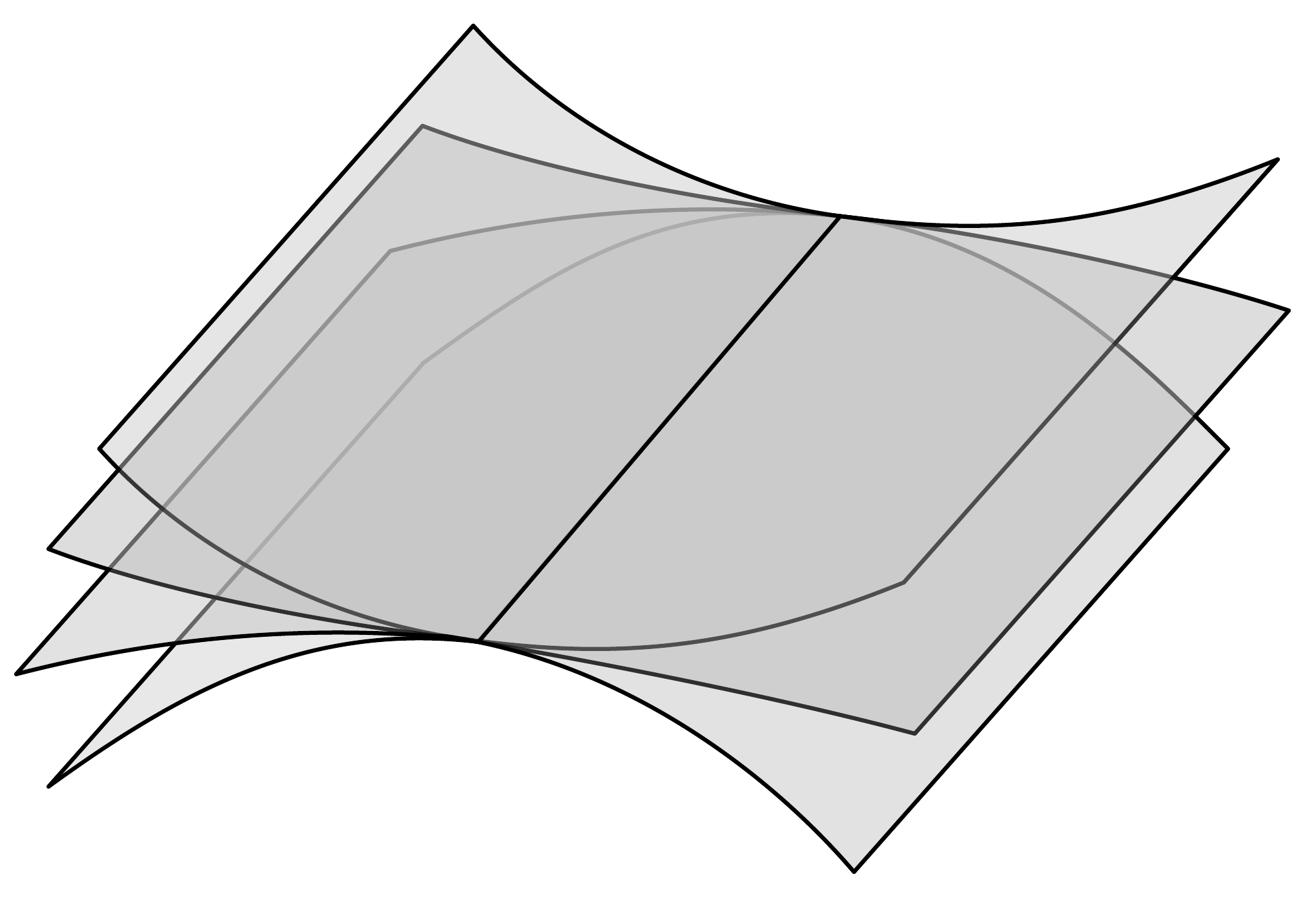}
\caption{The local branching near an edge $b$ of $\tau$. The two $\pi$-angles are along the tetrahedra above and below $e$.}
\label{fig:branch-edge}
\end{center}
\end{figure}

\smallskip
We conclude this subsection by recording a basic property about the action on the completed flow space $\mc P$, which following fairly easily from the fact that $\mc P$ contains no perfect fit rectangles. For a similar statement, which implies the lemma for nonsingular points, see \cite[Lemma 2.13]{tsang2022constructing}.

\begin{lemma}\label{lem:npfs}
Each nontrivial $g \in \pi_1(M)$ fixes at most one point of $\mc P$.
\end{lemma}

\begin{proof}
First, note that $\mc P$ itself has no perfect fit rectangles. This was observed in \cite[Section 5.3.1]{landry2025transverse}, but follows immediately from the fact that a perfect fit rectangle in $\mc P$ would project to a perfect fit rectangle in $\orb$ whose interior is disjoint from $\kappa$. 

Now if $g$ fixes a point $x$ in $\mc P$, then it either has a power $g^k$ fixing the half-leaves at $x$ or $x$ is a completion point and $g$ cyclically permutes the quadrants at $x$.
In the second case, it is clear that no other points are fixed. In the first case,
 it follows easily from the dynamics at $x$ together with the fact that $\mc P$ has no perfect fits that $g$ fixes no points other than $x$. Indeed, otherwise we could find a regular leaf $l$ separating $x$ from another fixed point $y$ so that $l$ intersects a half-leaf $r$ based at $x$. Then we consider the sequence $g^{kn} \cdot l$ as $n \to \infty$. On the one hand, each leaf in this sequence separates $x$ from $y$ and so the sequence cannot exit compact sets of $\mr {\mc P}$. On the other hand, up to replacing $g$ with its inverse, the points $g^{kn} \cdot l \cap r$ exit the end of $r$. This implies that the leaves $g^{kn} \cdot l$ accumulate on
a nonempty chain of leaves one of which makes a perfect fit with $r$; a contradiction.
\end{proof}

\subsection{Ladderpoles}\label{sec:ladderpole}
An abstract veering triangulation can also be characterized by the particular way that $\tau$ enters the ends of $M_0$.
To describe this, we truncate the ends of $M_0$ to get a compact manifold $\mr M$ with boundary. More precisely, take a small regular neighborhood $U_\omega$ of each orbit $\omega$ in $\kappa$, so that $\tau^{(2)}$ intersects $\partial U_\omega$ transversely, and that $\tau^{(2)}\cap \ol{U_\omega}$ is diffeomorphic to the product $(\tau^{(2)}\cap\partial U_\omega)\times[0,+\infty)$. The union $U$ of all $U_\omega$ is called a \emph{tube system} for $\kappa$, and we denote $M\ssm U$ by $\mr M$. The effect of the truncation on each tetrahedron is creating a \emph{tip} for each ideal vertex (see \Cref{fig:taut-tetrahedron} where the tips are colored in purple and green). Each tip is viewed as a triangle with two cusps and one smooth vertex. From now on, we always consider the truncated model of $\tau$, which is $\tau\cap \mr M$, and we continue to denote it by $\tau$.

For any $U_\omega$ in $U$, the intersection $\tau_{\omega}:=\tau^{(2)}\cap \partial U_\omega$ is a train track on a torus with complementary components the tips of the truncated tetrahedra. The train track $\tau_{\omega}$ has a coorientation coming from the coorientation of $\tau^{(2)}$.
We say a tip is \emph{upward} if it has two sides oriented outward, and \emph{downward} if otherwise.

The veering condition turns out to imply that there are finitely many disjoint closed train routes on $\tau_\omega$, called the \emph{ladderpole curves}, so that the following holds: the complement of the ladderpole curves are annuli called \emph{ladders}, and they alternate between downward ladders and upward ladders. Here, a ladder is downward if it is a union of downward tips, and similarly for upward. An example of the ladders on a boundary torus is illustrated in \Cref{fig:prongs}. See \cite{Landry_norm} for details.

Given a closed orbit $\omega$ in $\kappa$, let $\FF^s(\omega)$ be the periodic stable leaf containing $\omega$. Then the component of $\FF^s(\omega)\cap U_\omega$ containing $\omega$ intersects $\partial U_\omega$ in a collection of closed parallel curves, which we call the \emph{stable prong curves}. Similarly, we can define \emph{unstable prong curves} on $U_\omega$. By \cite[Lemma 3.2]{landry2025transverse}, on $\partial U_\omega$, each upward ladder contains exactly one stable prong curve, and each downward ladder contains exactly one unstable prong curve (see \Cref{fig:prongs}).

\begin{figure}[h!]
\begin{center}
\includegraphics[width = 0.3\textwidth]{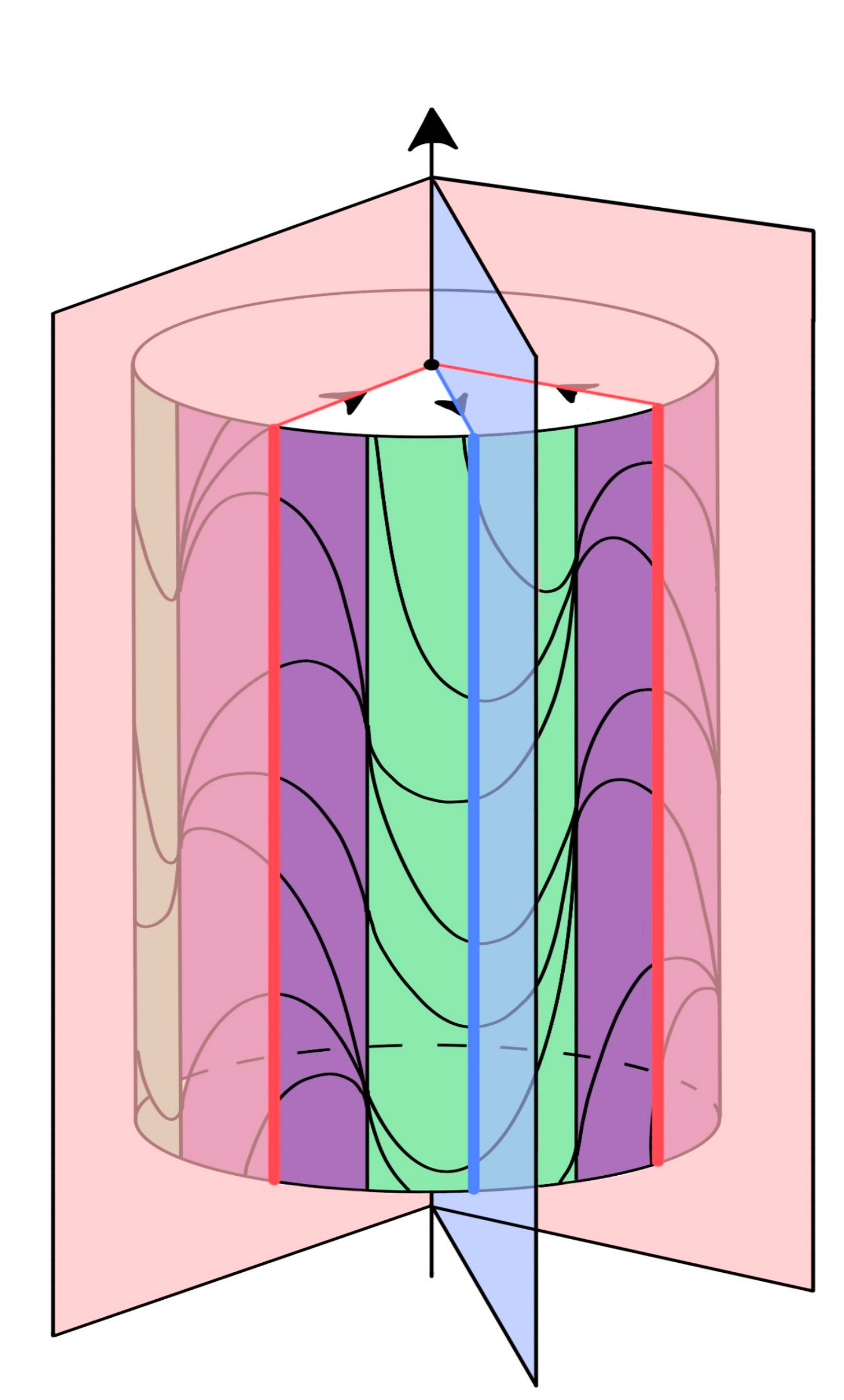}
\caption{Correspondence between prong curves and ladders.}
\label{fig:prongs}
\end{center}
\end{figure}

\subsection{The dual graph and branch cycles}\label{sec:cycles}

The \emph{dual graph} $\Gamma$ of the veering triangulation $\tau$ can be realized as an embedded $4$-valence directed graph in $\mr M$, with the orientation on each edge determined by the coorientation of the face of $\tau^{(2)}$ it crosses. By the combinatorics of a taut triangulation, every vertex has two incoming edges and two outgoing edges. An oriented loop of $\Gamma$ is called a \emph{dual cycle}. For a periodic sequence of tetrahedra that intersect the boundary of a tube in a ladderpole, the corresponding dual cycle $\gamma$ is called a \emph{branch cycle}. This gives a 1-to-1 correspondence between branch cycles and prong curves on $\partial U$.

\subsection{Transverse surfaces}\label{sec:ts}

For a pseudo-Anosov flow $\phi$, we say an embedded closed surface $S$ in $M$ is \emph{almost transverse} to $\phi$ if there is a dynamic blowup $\phi^\#$ of $\phi$ so that $S$ can be isotoped to be transverse to $\phi^\#$. We say the blowup is \emph{minimal} if $S$ intersects every blowup annulus in its core curve. This is equivalent to the condition that
there is no blowup of $\phi$ transverse to $S$ up to isotopy with fewer blowup annuli. 
If $S$ is almost transverse to $\phi$, then $S$ is incompressible, and the class 
$[S]\in H_2(M,\Z)$
has nonnegative algebraic intersection number with every closed orbit of $\varphi$.
Conversely, Mosher \cite{Mos92, mosher1992dynamical} shows that any such class
is represented by an almost transverse surface. Moreover,  the set of almost transverse surfaces, up to isotopy, is characterized in \cite{landry2025transverse}.

An embedded closed surface in $M$ is \emph{relatively carried} by $\tau$ if after an isotopy of $S$, the surface $\mr S := S\ssm U$ is carried by the branched surface $\tau^{(2)}$, and every component of $S\cap U$ is either a meridian disk of $U_\omega$ or a \emph{ladderpole annulus} in $U_\omega$, which is a $\pi_1$-injective annulus with both boundary components ladderpole curves. We call this embedding a carried position of $S$. It is shown in \cite[Theorem E]{landry2025transverse} that an embedded closed surface in $M$ is almost transverse to $\phi$ if and only if it is relatively carried by $\tau$. 

For a closed surface $S$ carried by $\tau$, we can put $S$ in a carried position so that it minimizes the number of ladderpole annuli. This is called \emph{efficient position} in \cite{landry2025transverse}. In particular, there are no ladderpole annuli with adjacent ladderpole boundary components since these can be pulled off using annuli moves (see \Cref{fig:annulus_move}). For more details, see \cite[Section 2.9]{landry2025transverse}.
In practice, we will always assume a closed transverse surface is already in efficient position.

\begin{figure}[htbp]
\begin{center}
\includegraphics[width = .8 \textwidth]{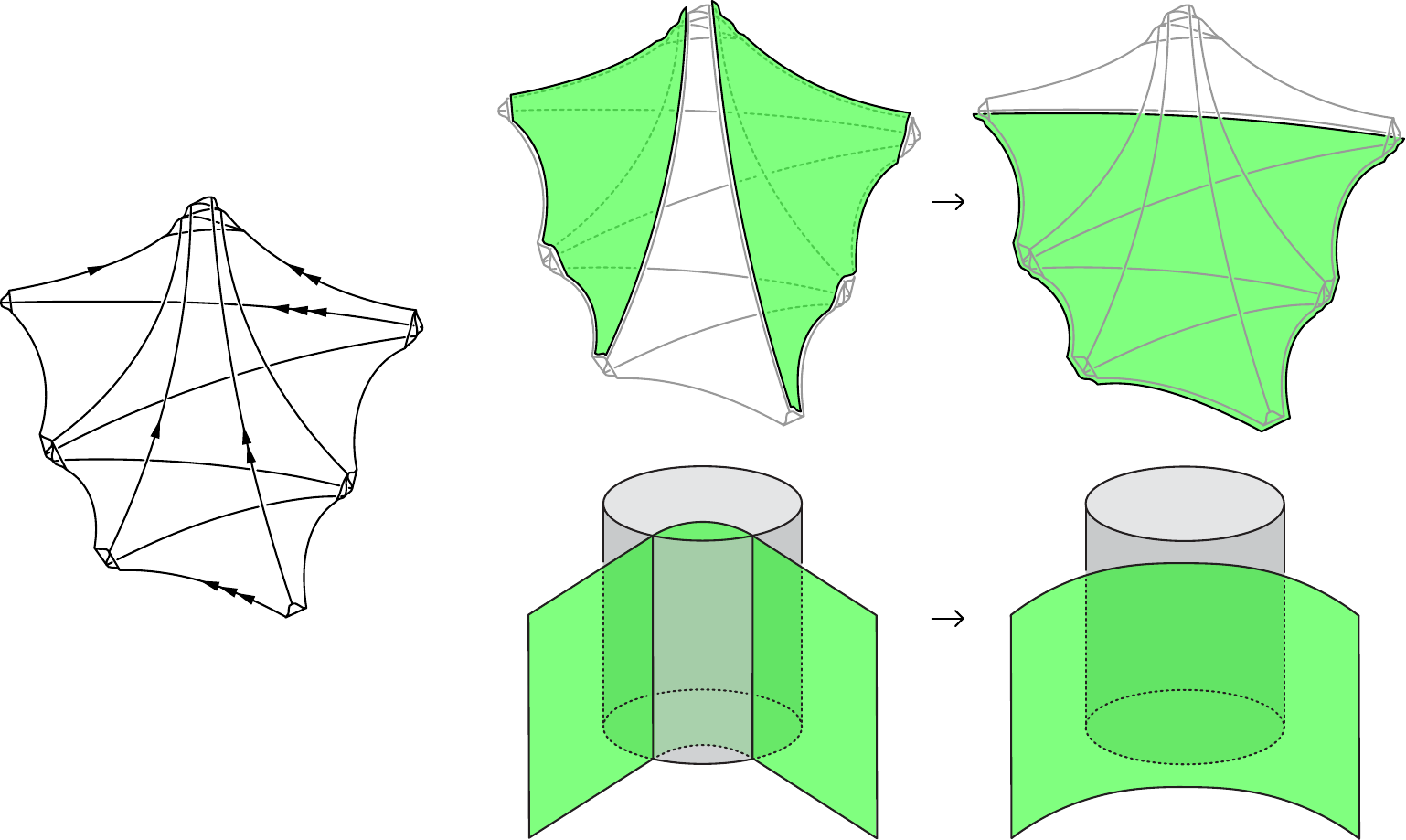}
\caption{An annulus move to remove two adjacent ladderpole components. Figure reproduced from \cite{landry2025transverse}.}
\label{fig:annulus_move}
\end{center}
\end{figure}

\bigskip

\subsection{Veering triangulations for almost pseudo-Anosov flows}\label{sec:apA}
Although the construction of $\tau$ is stated for pseudo-Anosov flow, the case of an almost pseudo-Anosov flow is essentially the same, as we remark below. The reader can find the complete details in \cite[Section 3]{landry2025transverse}.

Let $\phi$ be an almost pseudo-Anosov flow in $M$, and let $\phi^\flat$ be the pseudo-Anosov blowdown. Let $\kappa^\flat$ be a union of closed orbits of $\phi^\flat$ satisfying the conditions to construct a veering triangulation $\tau$ on $M\ssm\kappa^\flat$. For $\phi$, we define $\kappa$ to be the blowup of $\kappa^\flat$, i.e. the preimage of $\kappa^\flat$ under the blowdown map $M \to M$. Then $\kappa$ is a union of all blowup complexes, all singular orbits and a finite number of regular closed orbits killing perfect fits. There is a homeomorphism from $M\ssm\kappa$ to $M\ssm\kappa^\flat$ inducing an orbit equivalent from $\phi|_{M\ssm\kappa}$ to $\phi^\flat|_{M\ssm\kappa^\flat}$, so $\tau$ can be realized as a veering triangulation of $M\ssm\kappa$ so that it is transverse to $\phi$ with the same properties as established above.

One can also recover $\tau$ by running the Agol-Gu\'eritaud construction for $\phi$. The orbit space $\mr \pp$ for $\phi$ on $M\ssm\kappa$ can be identified with the orbit space $\mr \pp^\flat$ for $\phi^\flat$ on $M\ssm\kappa^\flat$, and we can complete $\mr \pp$ to the completed orbit space by the same completion $\pp$ of $\mr \pp^\flat$ by adding ideal points.

When $\phi$ is almost pseudo-Anosov, a tube system $U$ for $\phi$ is a union of tubular neighborhoods of components in $\kappa$, and it can be identified with a tube system for $\phi^\flat$ using the blowdown map. If $\omega$ is a component of $\kappa$, let $\FF^{s}(\omega)$ be the stable leaf containing $\omega$. The stable prong curves on $\partial U_\omega$ is again defined to be the intersection of $\partial U_\omega$ with the component of $\FF^{s}(\omega)\cap U$ containing $\omega$. Similarly we can define the unstable prong curves.  The discussions in \Cref{sec:ts,sec:cycles,sec:ladderpole} continues to hold true for $\mr M:=M\ssm U$.

\section{The core branched surface in $\mr M \cut \mr S$}
\label{sec:branched}

For the setup, let $M$ be a closed atoroidal 3-manifold with an almost pseudo-Anosov flow $\phi$, and $S$ is a closed embedded surface transverse to $\phi$. Let $\kappa$ be a subset of $M$ containing all singular orbits, blowup complexes and a finite collection of regular orbits that kill perfect fits. Let $\tau$ be the veering triangulation of $M \ssm \kappa$ associated to $\phi$. In this section, we use the carried position of $S$ on $\tau$ to analyze the structure of $M \cut S$. In particular, we construct a branched surface $B$, called the \emph{core branched surface} 
which is the main tool to construct properly embedded transverse surfaces in \Cref{sec:transverse}. Roughly speaking, $B$ captures the parts in $\tau^{(2)}$ that are not witnessed by $S$.

In what follows, we set $N = M \cut S$ for notational simplicity, but we still use the notation $M\cut S$ when we wish to emphasize the role played by the surface $S$.

\subsection{Cutting along a carried surface}\label{sec:cut}

Recall that we consider $\tau$ as its truncated model living inside $\mr M$.
Since $S$ is almost transverse to $\varphi$, it can be put in a relatively carried position with respect to some tube system $U$ of $\kappa$ (see \Cref{sec:veering}). Recall from \Cref{sec:ts} that $\mr S$ is defined to be $S\cap \mr M = S\ssm U$.

\begin{figure}[h]
    \centering
    \labellist
    \pinlabel $\tau^{(2)}$ at 20 500
    \pinlabel $W$ at -40 150
    \pinlabel $\ol{\tau}$ at 800 220
    \endlabellist
    \includegraphics[width=4in]{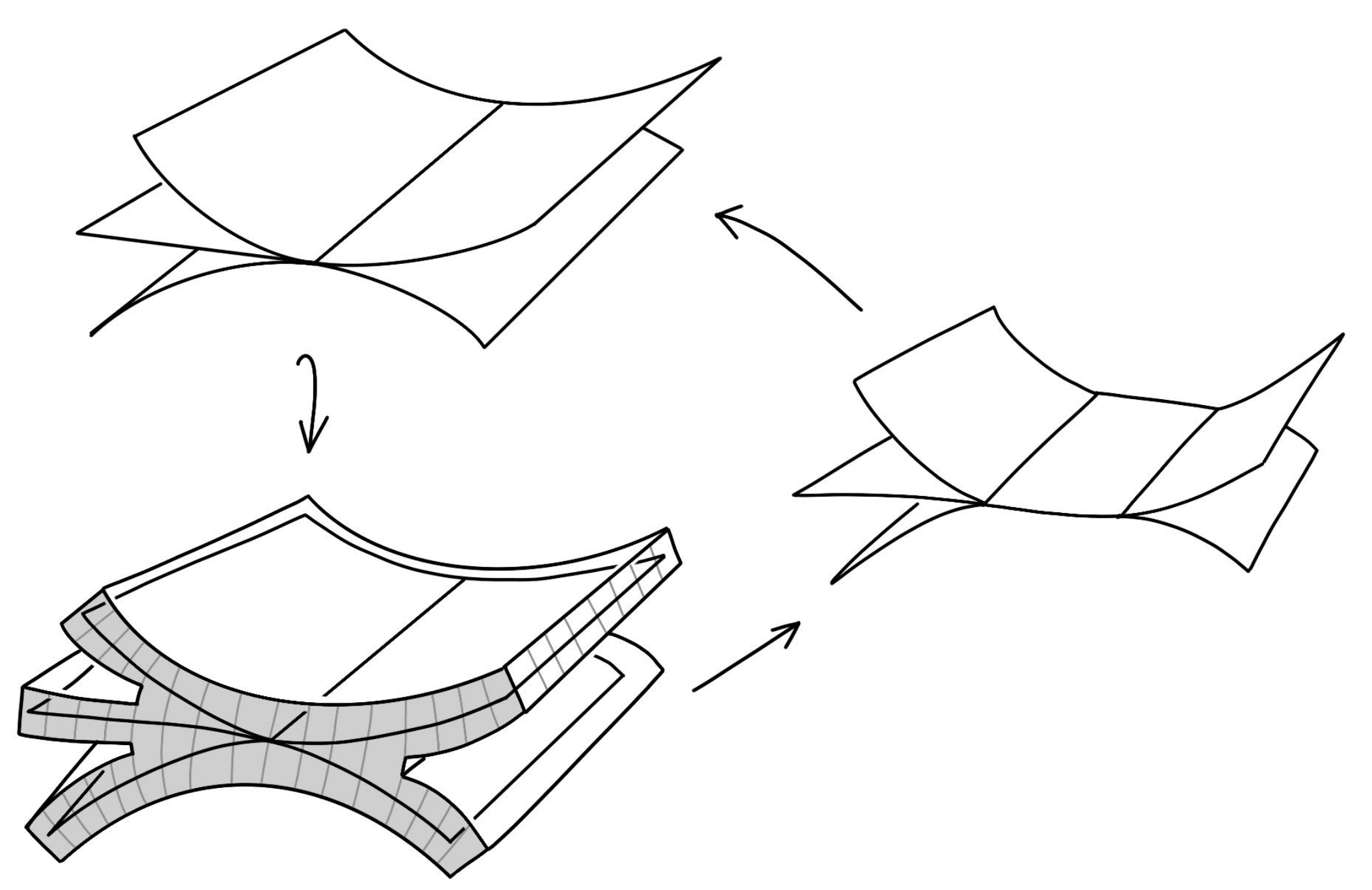}
    \caption{The branched surface $\tau^{(2)}$ can be embedded into the branched surface fibered neighborhood $W$ transverse to the fibers. Collapsing the fibers in $W$ produces $\ol{\tau}$, and collapsing rectangles in $\ol{\tau}$ produces $\tau^{(2)}$. }
    \label{fig:tau-bar}
\end{figure}

To be more precise, we first fix a branched surface fibered neighborhood $W \subset \mr M$ 
of $\tau^{(2)}$. Recall this means that $W$ is foliated by intervals such that the composition of the inclusion $\tau^{(2)} \hookrightarrow W$, which embeds $\tau$ transverse to the fibers, together with the map $W \twoheadrightarrow \tau^{(2)}$ that collapses the fibers is near the identity; see \cite[Section 1]{floyd1984incompressible}. Since $\tau^{(2)}$ is transverse to the flow, outside of a small neighborhood of $\partial \mr M$, we may take the $I$-fibers of $W$ to be flow segments of $\varphi$ .

Technically speaking, the branched surface we get by collapsing the fibers of $W$ is a slightly different branched surface $\ol{\tau}$, obtained from $\tau^{(2)}$ by replacing every edge with a rectangle (\Cref{fig:tau-bar}). Note that $\ol{\tau}$ has only one-sided branching at each arc in its branch locus, which is assumed in the discussion of \cite{floyd1984incompressible}. However, to avoid additional technicalities, for now we ignore the rectangles in $\ol{\tau}$, and consider $\tau^{(2)}$ as the base of $W$. The quotient map $W\to\tau^{(2)}$ is the composition of the fiber-collapsing map $W\to\ol{\tau}$, with the map $\ol{\tau}\to\tau^{(2)}$ that collapses rectangles to edges.

For each component $\omega$ of $\kappa$, the associated component $U_\omega$ in $U$ is a solid torus and the components of $U_\omega \cap S$ are either a collection of disks or a collection of ladderpole annuli (see \Cref{sec:ts}). We realize $S$ in relatively carried position so that, after an isotopy in $M$, $\mr S \subset \mr M$ is properly embedded, contained in $W$ where it is transverse to the $I$-fibers, and is obtained from $S$ by removing disks and annuli in $U\cap S$. We say $U_\omega$ is a \emph{disjoint, meridian, or ladderpole} tube if $U_\omega\cap S$ is empty, a collection of meridian disks, or a collection of ladderpole annuli, respectively.

We now examine more closely the structure of $\mr M \cut \mr S$, which we denote by $\mr N$.
The boundary $\partial \mr N$ consists of the following parts:
\begin{enumerate}
    \item two copies of $\mr S$, which we denote by $\partial_+\mr N = \mr S^+$ and $\partial_-\mr N = \mr S^-$ with $\phi$ flowing outwards on $\mr S^+$ and inwards on $\mr S^-$;
    \item annuli coming from cutting the boundary of meridian tubes by $\mr S$. For simplicity, we will adjust $\partial \mr N$ via small proper isotopy in $N$ so that these annuli are flow parallel. Here, the isotopy is small in the sense that it is supported in a neighborhood of the ends of $(M\ssm\kappa)\cut S$ where the triangulation looks like a product;
    \item annuli coming from cutting the boundary of ladderpole tubes by $\mr S$. Such annuli can be further divided into two categories: those with one boundary on $\mr S^+$ and the other one on $\mr S^-$, and those with both boundary components are on $\mr S^+$ or $\mr S^-$;
    \item tori that are the boundary of disjoint tubes.
\end{enumerate}

The cut subspace $W \cut \mr S \subset \mr N$ inherits a fibered structure from $W$ since $\mr S$ is transverse to the fibers. Let $W_I$ be the $I$-bundle defined as the union of fibers that run between $\mr S^-$ and $\mr S^+$. Roughly speaking, this is the part where $\mr S$ runs parallel to itself in the carried position. Collapsing the $I$-fibers of $W_I$ gives a surface $Y$. By construction $Y$ is naturally carried by $\ol \tau$ and so it has a cell structure induced from that of $\ol \tau$. To simplify our discussion, we alter the cell structure as follows: for any component of $Y$ that is not a rectangle mapping homeomorphically to a rectangular cell of $\ol \tau$ 
we instead take the cell structure induced by further projecting to $\tau^{(2)}$. In particular, $Y$ is tessellated by truncated triangles, except for components which are rectangles that correspond to edges of $\tau^{(2)}$. 

\smallskip

We introduce some notations to describe this cellular structure on $Y$.
For a 2-cell of $\ol \tau$ (either a rectangle or a truncated triangle), an edge of it is \emph{long} if it is in the branched locus, and is \emph{short} if it is in $\partial \ol \tau$. Then we can divide the 1-cells in $Y$ into short and long edges naturally, since each $2$-cell of $Y$ corresponds to either a truncated triangular sector or a rectangular sector of $\ol \tau$.

Recall that $W_I$ is an $I$-bundle over $Y$. We define the \emph{vertical boundary} of $W_I$, denoted by $\partial_v W_I$, to be the collection of annuli on $\partial W_I$ fibering over $\partial Y$. The subset of $\partial_v W_I$
fibering over a long edge of $\partial Y$ is called a \emph{long rectangle}, and similarly we define \emph{short rectangles} to be the parts fibering over short edges of $\partial Y$. Every boundary component of $Y$ is either \emph{totally short}, consisting completely of short edges (which corresponds to a component of $\partial_v W_I$ contained in $\partial U$), or is \emph{mixed}, being a cycle alternating between long edges and finite sequences of short edges (see \Cref{fig:Y} for an example). We denote the union of totally short components by $\partial_{\text{ts}} Y$, and the union of mixed components by $\partial_{\text{mix}} Y$. Similarly we also have $\partial_\text{ts} W_I$ and $\partial_\text{mix} W_I$, which are respectively the part of $\partial_v W_I$ fibered over $\partial_\text{ts} Y$ and $\partial_\text{mix} Y$, and are called totally short annuli and mixed annuli respectively.

\begin{figure}[h]
\begin{center}
\includegraphics[width = .4 \textwidth]{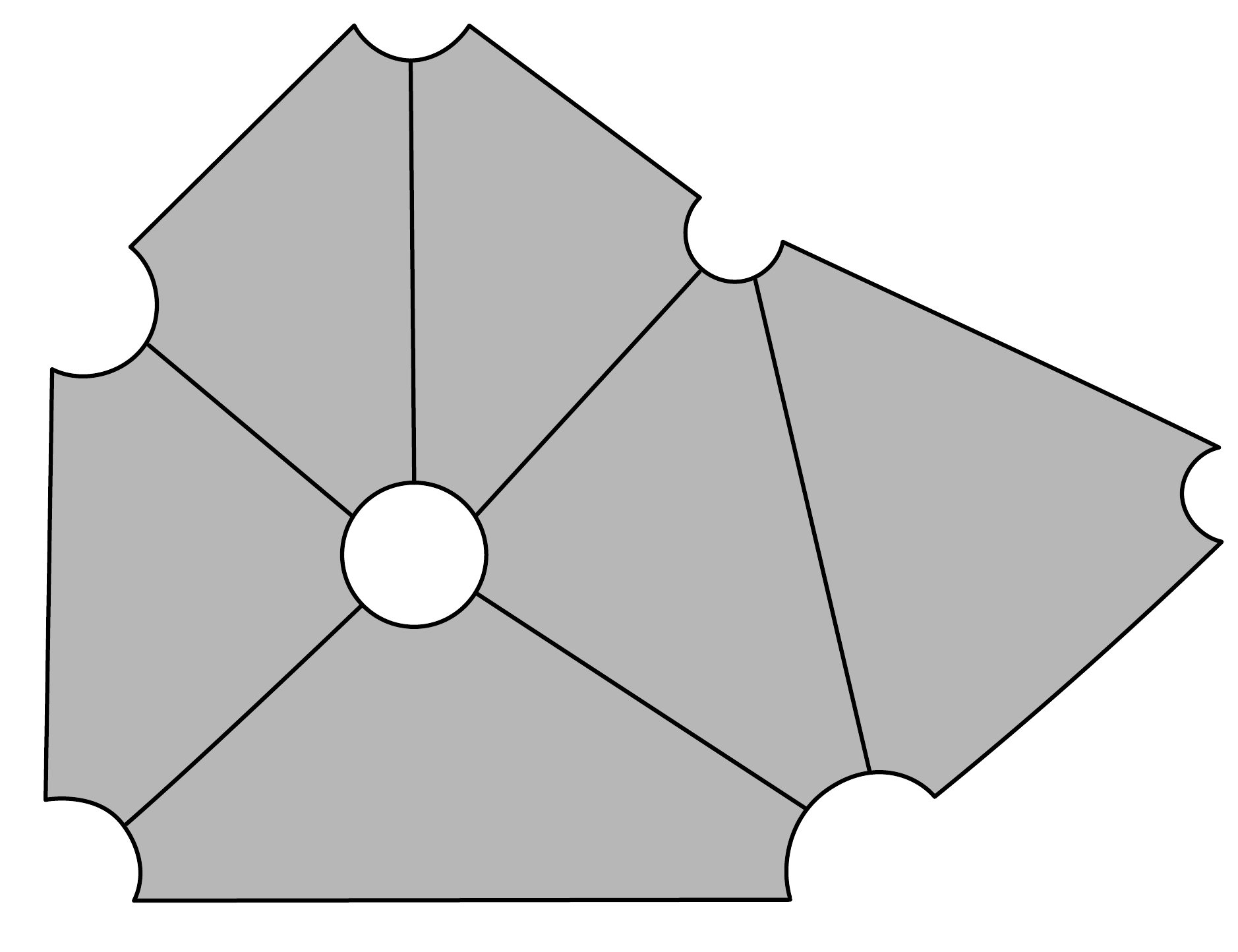}
\caption{A possible shape of a component of $Y$ where $\partial Y$ has one mixed (outer) component and one totally short (inner) component.}
\label{fig:Y}
\end{center}
\end{figure}

\smallskip

\subsection{The core branched surface $B \subset \mr M \cut \mr S$}\label{sec:construct-core}
Define the submanifold $\mr M_{\mr S}=\mr N\cut W_I$ and note that $(W \cut \mr S) \cut W_I$ is itself a branched surface fibered neighborhood in $\mr M_{\mr S}$ of a branched surface $B'$. This is to say that no fiber of $(W \cut \mr S) \cut W_I \subset \mr M_{\mr S}$ runs between the boundary components $\partial_\pm \mr M_{\mr S}$ and so collapsing $(W \cut \mr S) \cut W_I$ along these fibers (and keeping endpoints along $\partial_\pm \mr M_{\mr S}$ 
fixed) produces a properly embedded branched surface $B'$ in $\mr M_{\mr S}$. For details, see \cite[Section 1]{floyd1984incompressible} as well as the construction of the `fiber branched surface' in Section 4 of \cite{oertel1984incompressible}.
Technically, to follow the description in these references verbatim, one should first double $\mr M_{\mr S}$ along $\partial_\pm \mr M_{\mr S}$, apply the collapsing map, then restrict back to $\mr M_{\mr S}$, but the end result is the same.

The sectors of $B'$ are in bijection with the sectors of $\tau^{(2)}$ that are \emph{not} traversed by $\mr S$. More precisely, the composition $B' \hookrightarrow (W \cut \mr S) \cut W_I \subset W \twoheadrightarrow \tau^{(2)}$ induces the desired bijection and is near the identity on each sector.
Moreover, each long edge of each such sectors is either interior to $B'$ (i.e. paired with other long edges of sectors in $B'$), smoothed into $\partial_\pm (\mr N)$, or contained in a long rectangle in $\partial_{\mathrm{mix}} W_I$. 
The cases depend on how $\mr S$ traverses the associated edge of $\tau^{(2)}$.
See \Cref{fig:B'}.

\begin{figure}[h]
    \centering
    \includegraphics[width=4in]{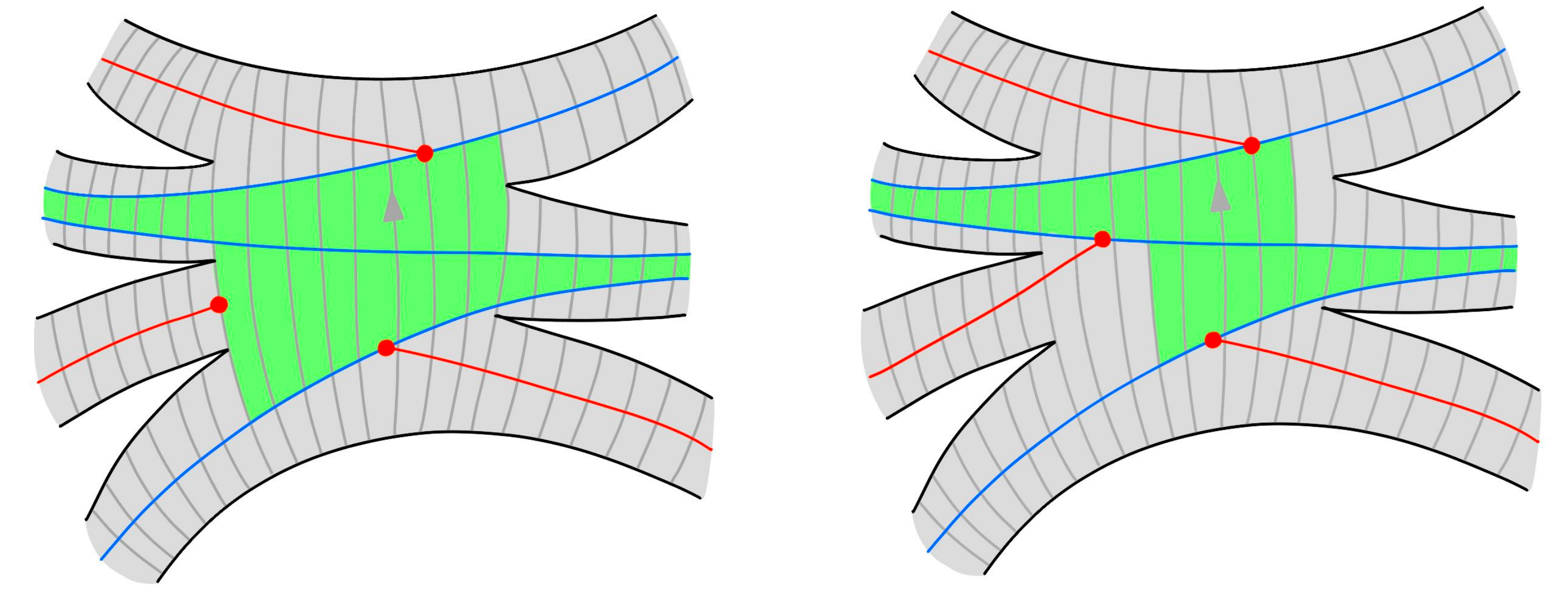}
    \caption{Constructing $B$ and $B'$ in $\mr M_{\mr S}$. This is a local picture in $W$ with the $I$-fibers drawn in darker gray. The flow direction is upward. The carried surface $\mr S$ is blue, $W_I$ is shaded green, and $B'$ is red on the left. The right figure indicates the modification of $W_I$ (see \Cref{rmk:adjust}) and the construction of $B \subset \mr M \cut \mr S$ from $B'$, with $B$ drawn in red.}
    \label{fig:B'}
\end{figure}

We next modify $B'$ to obtain a branched surface that is properly embedded in $\mr N$ (whereas $B'$ is properly embedded in $\mr M_{\mr S}$). Note that by construction $\partial B'$ can meet the long rectangles in $\partial_\text{mix} W_I$. (See for example the leftmost black dot in \Cref{fig:B'}.)
To fix this, we push $\partial B'$ on these rectangles slightly into $W_I$ and vertically up to $\mr S^+$ while keeping the branched surface embedded in $W$ and transverse to the fibers of $W$. We call the new branched surface $B$ and note that $B$ is properly embedded in $\mr N$. Since this isotopy can be done transverse to the fibers of $W$, $B\cap\partial U$ remains carried by the train track $\tau^{(2)}\cap\partial U$. 

To summarize: the branched surface $B$ is properly embedded in $\mr N$, 
transverse to the fibers (of $W$) and co-oriented by the flow direction; in particular, $\partial B\subset \mr S^+\cup \mr S^-\cup (\partial U\cut S)$.
We remark that by construction, the branch locus of $B$, like that of $B'$, occurs only at edges of $\tau^{(2)}$ that are \emph{not} traversed by $\mr S$. 

\begin{remark}[Adjusting $\mr M_{\mr S}$]\label{rmk:adjust}
Since portions of $\partial B'$ were pushed slightly into $W_I$ during the construction of $B$, the branched surface $B$ itself is not totally contained in $\mr M_{\mr S}$. Since such containment will be useful moving forward (see \Cref{lem:surface-represents-class}), it can be achieved by slightly modifying $\mr M_{\mr S}$ as follows: first recall that $W_I$ is an oriented $I$--bundle over $Y$. Let $Y' \subset Y$ be a compact subsurface isotopic to $Y$ obtained by pushing the long edges of $Y$ into its interior, keeping the short edges fixed,
and let $W_I'$ be the subbundle of $W_I$ over $Y'$. 
Then the isotopy from $B'$ to $B$ can be done in the complement of $W_I'$ so that $B$ is contained in $\mr N \cut W_I'$, while still properly embedded in $\mr N$. See also \Cref{fig:B'}. Moving forward, we now set $\mr M_{\mr S} = \mr N \cut W_I'$ so that $B \subset \mr M_{\mr S} \subset \mr N$. 
\end{remark}

\subsection{An alternative description of $B$}

To make the construction clear, we describe an alternative description of the core branched surface $B$. In $\mr N$, we push $W\cut \mr S$
off $\mr S^\pm$ and call the resulting submanifold $W'$. Then $W'$ is a branched surface neighborhood of an embedded branched surface $B^*$ in $\mr N$. The branched surface $B^*$ is constructed by first collapsing the $I$-fibers of $W'$ and then collapsing the rectangles to edges, similar to how we get $\tau^{(2)}$ from $W$ and $\ol{\tau}$ in \Cref{fig:tau-bar}. After embedding $B^*$ in $W\subseteq\mr M$, the quotient map $W\to\tau^{(2)}$ restricted on $B^*$ maps sectors to sectors up to small error (the error is the same as why $\tau^{(2)}\to W\to\tau^{(2)}$ is only near the identity). Therefore, each sector in $B^*$ has a unique sector in $\tau^{(2)}$ associated to it.
\par
In $\mr N\cut B^*$, there are two product neighborhoods $\mr S^\pm\times[0,1]$ of $\mr S^\pm$. In other words, $\mr S^+$ is carried by the `topmost' subsurface $B^+$ in $B^*$, and similarly $\mr S^-$ is carried by the `bottommost' subsurface $B^-$ in $B^*$. Both $B^+$ and $B^-$ are unions of sectors in $B^*$. The sectors in $B^-$ and $B^+$ correspond to sectors in $\tau^{(2)}$ traversed by $\mr S$. The intersection of $B^+$ and $B^-$ can be non-empty because $S$ can traverse an edge or a sector. The intersection, if non-empty, is a union of 1- and 2-cells of $B^*$,
and the union of the cells correspond to the surface $Y$ defined above.
Note again that a component of $B^+\cap B^-$ consisting of a single edge corresponds to a rectangular component of $Y$.
\par
We next modify $B^*$ to obtain the branched surface $B$ that is properly embedded in $\mr N$. First, we consider the sectors in $B^*$ that are not in $B^+\cup B^-$, and denote the union of these sectors by $B^*_m$ ($m$ for middle). The boundary of $B^*_m$ consists of edges contained in $B^-$ and $B^+$. For edges in $B^-$ but not in $B^+$, we push them to $\mr S^-$ and smoothen out according to the coorientation. 
Similarly, for edges in $B^+$ but not in $B^-$, we push them to $\mr S^+$ and smoothen out. After this step, we get an intermediate branched surface, which is exactly $B'$ in \Cref{sec:construct-core}. 
Finally, for edges in $B^-\cap B^+$, we push them up to $\mr S^+$ as well. After this operation we obtain a properly embedded branched surface, which is the core branched surface $B$, in $\mr N$ transverse to $\phi$.

\medskip

We are now ready to start the construction to prove \Cref{thm:properly-embedded-tst}. The proof will take up \Cref{sec:nonneg-class} and \Cref{sec:fill}.

\section{From nonnegative classes to transverse surfaces in $\mr M \cut \mr S$} \label{sec:nonneg-class}

Let $\rm{Orb}(M\cut S)$ denote the collection of closed orbits of $\varphi$ in $M \cut S$, where $\phi$ is an almost pseudo-Anosov flow transverse to $S$.
As in the introduction, a class $\eta \in H^1(M\cut S)$ is \emph{nonnegative} (resp. \emph{positive}) if it is nonnegative (resp. positive) on the oriented orbits of $\rm{Orb} (M \cut S)$. We remark that $\rm{Orb} (M \cut S)$ as a set of isotopy classes of oriented closed loops (hence also the definition of nonnegative/positive classes) is independent of the choice of the transverse blowup $\phi$. This is because the combinatorial type of the blowup complexes of a minimal blowup transverse to $S$ is unique by \cite[Theorem D]{landry2025transverse}, and any other almost pseudo-Anosov flows transverse to $S$ are further blowups of the minimal one.

The main goal of this section and the next section is to prove \Cref{thm:transverse-surface}, which we restate in the following form.

\begin{theorem}\label{thm:properly-embedded-tst}
Let $\varphi$ be an almost pseudo-Anosov flow on a closed atoroidal $3$-manifold $M$ and let $S$ be a surface transverse to $\varphi$. Suppose $\varphi$ is minimally blown up with respect to $S$. Then for any nonnegative class $\eta \in H^1(M\cut S)$, there is a properly embedded essential surface $\Sigma \subset M\cut S$ representing $\eta$ and a dynamic blowup of $\phi$ transverse to both $S$ and $\Sigma$.
\end{theorem}

In this section, we will show as the first step that for any nonnegative $\eta\in H^1(M\cut S)$, we can construct a properly embedded surface $\mr \Sigma \subset \mr M \cut \mr S$ that is carried by the core branched surface $B$. Moreover, as an embedded surface of the 
submanifold $\mr M_{\mr S}$ (see \Cref{rmk:adjust})
it represents the pull-back of $\eta$ in $\mr M_{\mr S}$. This is achieved in \Cref{lem:surface-represents-class}. We will then extend this surface to be properly embedded in $M \cut S$ in \Cref{sec:transverse}.

\smallskip

Recall the dual graph $\Gamma$ of the triangulation $\tau$ is a directed graph embedded in $\mr M$ whose edges are oriented by the coorientation on the faces of $\tau$, which themselves are positively transverse to $\varphi$. Let $\mc C_\Gamma$ be the set of directed cycles of $\Gamma$, called \emph{dual cycles}, and note that each closed path in $M \ssm \kappa$ that is positively transverse to $\tau^{(2)}$ determines a unique dual cycle by considering the associated sequence of tetrahedra through which it passes.
Let $\mathrm{Orb}(M \ssm \kappa)$ be the set of closed orbits of $\varphi$ that are contained in $M \ssm \kappa$ (i.e. not those of $\kappa$) together with stable and unstable prong curves in each component of $\partial \mr M$. 
Recall from \Cref{sec:ladderpole} that for each component of $\partial \mr M$, each upward ladder contains exactly one stable prong curve and each downward ladder contains exactly one unstable prong curve. Also recall that each ladder determines exactly one 
dual cycle, which we call its \emph{branch cycle}.

The following is essentially proven in \cite[Section 6]{LMT21}, although there a more refined correspondence is proven using the \emph{flow graph}. Here, we need to work with the dual graph because of the relation to carried surfaces.

\begin{proposition}\label{prop:Phi}
There is a surjective map $\Psi \colon \mc C_\Gamma \to \mathrm{Orb}(M \ssm \kappa)$ such that
\begin{itemize}
\item if $\gamma$ is not a branch cycle, then $\Psi(\gamma)$ is the unique orbit homotopic to $\gamma$,
\item if $\gamma$ is a branch cycle in an upward/downward ladder of a cusp, then $\Psi(\gamma)$ is the unique stable/unstable prong curve contained in that ladder.
\end{itemize}
Moreover, in each case $\gamma$ is homotopic to $\Psi(\gamma)$ within $M\ssm \kappa$ through curves that are positively transverse to $\tau$. 
\end{proposition}

\begin{proof}
In order to confirm that all homotopies take place in $M \ssm \kappa$, rather than $M$, the argument takes place in the completed flow space $\mc P$ of $M \ssm \kappa$. See \Cref{sec:veering} or \cite[Section 4]{LMT21} for additional details. Also, recall from \Cref{lem:npfs} that each nontrivial element of $\pi_1(M)$ fixes at most one point in $\mc P$.

Following \cite[Section 6.3]{LMT21},
the dual graph $\Gamma \subset M\ssm \kappa$ has a preimage $\wt \Gamma$ in $\wt{ M \ssm \kappa}$ and for any dual cycle $\gamma$ we fix a lift $\wt \gamma$ to $\wt \Gamma$ and a $g\in \pi_1(M\ssm \kappa)$ that generates its stabilizer and translate $\wt \gamma$ in the positive direction. The dual line $\wt \gamma$ is associated to a sequence of maximal rectangles in $\mc P$ whose intersection is a single point $p \in \mc P$ such that $g \cdot p = p$.

The local structure of the veering triangulation around a cusp (see the discussion preceding Lemma 6.13 in \cite{LMT21}) shows that the point $p$ is a completion point of $\mc P$ if and only if the sequence of rectangles all contain $p$ in their boundary and this occurs if and only if $\gamma$ is a branch cycle. 

If $p$ is not a completion point of $\mc P$, then its preimage in $\wt{ M \ssm \kappa}$ is an orbit $\wt \omega$ that is also fixed by $g$. This determines a closed orbit $\omega = \wt \omega / \langle g \rangle \subset M \ssm \kappa$ that is homotopic to $\gamma$ in $M \ssm \kappa$. Since $\gamma$ is not a branch cycle, \cite[Proposition 6.7]{LMT21} states that there is a homotopy from $\gamma$ to $\omega$ in $\mr M$ through curves that are transverse to $\tau$.

If $p$ is a completion point of $\mc P$, then $\gamma$ is a branch cycle at the associated cusp. In this case, the maximal rectangles along $\wt \gamma$ all have $p$ in either their vertical or horizontal boundary because the sequence is $g$ invariant. Assume $p$ is in their vertical boundary. Then there is a unique unstable half-leaf at $p$ that meets the interior of each such rectangle (see \Cref{fig:branch-cycle}). The corresponding unstable half-leaf in $M$ intersects $\partial \mr M$ in an unstable prong curve $c$ in the associated downward ladder of $\tau$,
and we set $\Psi(\gamma) = c$. In particular, $c$ and $\gamma$ traverse the same sequence of tetrahedra and are therefore homotopic through transverse curves. This completes the proof.
\end{proof}

\begin{figure}[h]
    \centering
    \labellist
    \pinlabel $p$ at -30 650
    \pinlabel $l$ at 900 650
    \endlabellist
    \includegraphics[width=2in]{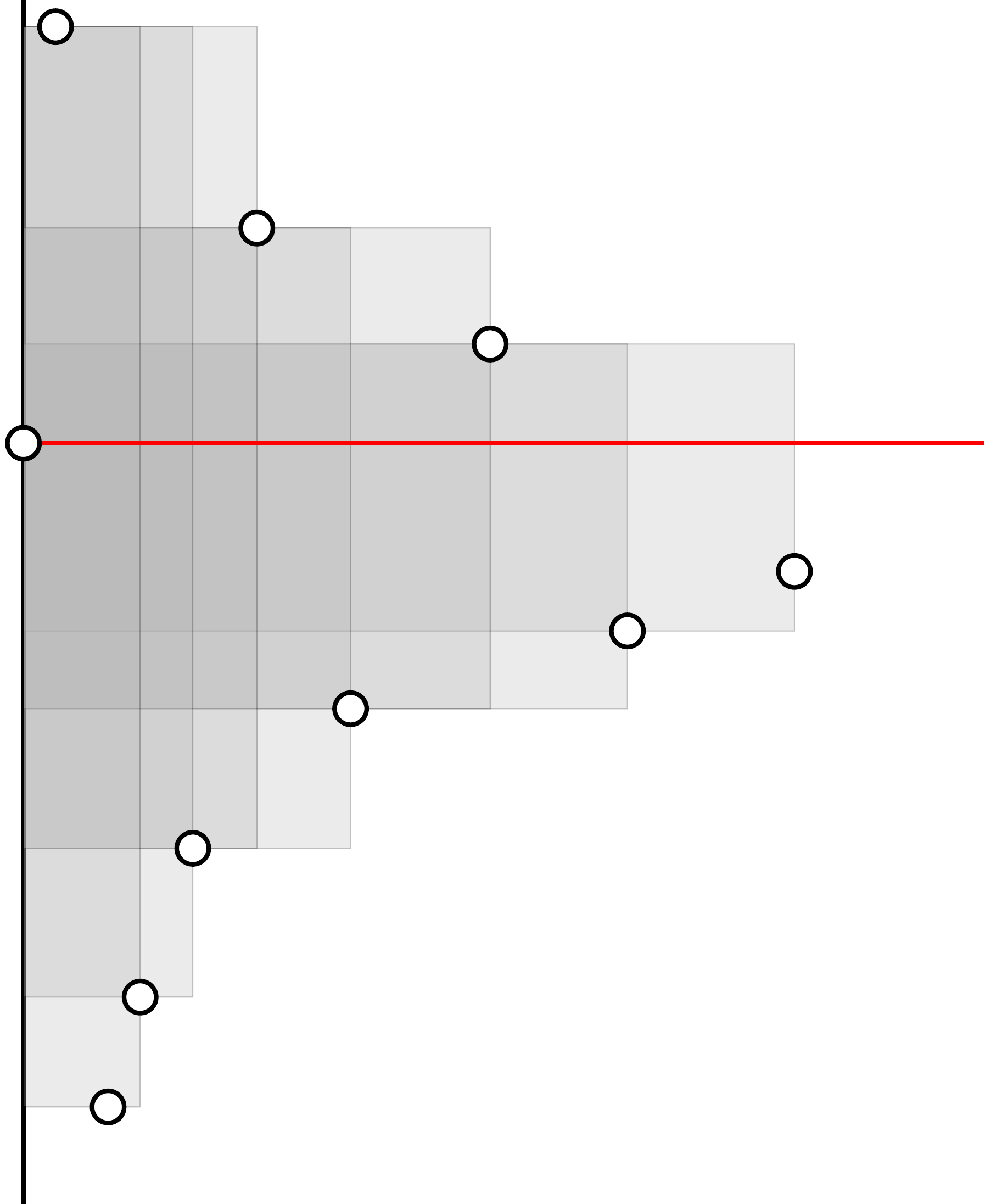}
    \caption{A sequence of tetrahedra rectangles that corresponds to a branch cycle.}
    \label{fig:branch-cycle}
\end{figure}

As in \Cref{sec:ts}, we put $S$ in efficient position once and for all. Since $S$ is relatively carried, $\Gamma$ intersects $S$ transversely in the interior of its edges. Let $\Gamma\cut S$ be the subgraph of $\Gamma$ obtained by removing the edges that intersect $S$. Then $\Gamma\cut S$ is embedded in $\mr N$, and the sectors of the core branched surface $B$ constructed in \Cref{sec:branched} are in one-to-one correspondence with edges in $\Gamma \cut S$ since each of these is in natural correspondence with the faces of $\tau^{(2)}$ that are not traversed by $\mr S$. Observe that $\Gamma \cut S \subset M \cut S$ and so $\eta$ can be pulled back to a class $i^*\eta \in H^1(\Gamma\cut S)$ via the inclusion $i:\Gamma\cut S\to M \cut S$.

\begin{corollary} \label{lem:pullback}
Each directed cycle of $\Gamma \cut S$ is homotopic \emph{in} $M \cut S$ to a closed orbit of $\varphi$, and each closed orbit of $\varphi$ in $M \cut S$ is homotopic \emph{in} $M \cut S$ to a directed cycle of $\Gamma \cut S$.

In particular, the class $i^*\eta$ is nonnegative on directed cycles of $\Gamma \cut S$. 
\end{corollary}

\begin{proof} As before, we set $N = M \cut S$. 

Let $\gamma$ be a directed cycle of $\Gamma \cut S$. First, assume that as a directed cycle of $\Gamma$, $\gamma$ is not a branch cycle. Then by \Cref{prop:Phi}, $\omega = \Psi(\gamma)$ is a closed orbit of $\varphi$ in $M \ssm \kappa$ that is homotopic to $\gamma$ in $M \ssm \kappa$ through curves that are positively transverse to $\tau$. Since $S$ is carried by $\tau$, this homotopy is disjoint from $S$ and therefore it is a homotopy within $N$. 

Next assume that $\gamma$ is a branch cycle associated to a component $T$ of $\partial \mr M$. Again by \Cref{prop:Phi}, $\Psi(\gamma)$ is a prong curve $c$, and $\gamma$ and $c$ are also homotopic through curves that are positively transverse to $\tau$. Let $U_\omega$ be the tube in $U$ with $\partial U_\omega = T$ where $\omega$ is a component of $\kappa$. By definition there is a half-leaf $\ell$ of $\phi$ based at a closed orbit $\omega_0\subseteq\omega$ such that $\ell \cap U$ contains an annulus $A$ from $\omega_0$ to $c$. Note that since $S$ does not meet $T = \partial U$ along adjacent ladderpole curve in efficient position, $S$ crosses the blowup annuli of $U$ but not the annuli between closed orbits and prong curves. In particular, $S$ does not cross $A$. Therefore, the prong curve $c$, and hence dual cycle $\gamma$, is homotopic to $\omega_0$ through curves that are disjoint from $S$. This shows that there is a homotopy from $\gamma$ to the closed orbit $\omega_0$ within $N$, completing the proof of the first statement.

The second statement is essentially the same. Let $\omega_0$ be a closed orbit in $N$. If $\omega_0$ is not in $\kappa$, then it crosses a sequence of faces of $\tau$ and therefore determines a dual cycle $\gamma$ that is not a branch cycle. Hence, \Cref{prop:Phi} gives that $\Psi(\gamma) = \omega_0$ and we conclude as above. If $\omega_0$ is an orbit in a component $\omega$ of $\kappa$, then just as above we let $U_\omega$ be the tube containing $\omega$ and consider the prong curve $c$ in $\partial U_\omega$ that cobounds a (non-blowup) annulus in $U_\omega$ with $\omega_0$ that itself is contained a stable/unstable leaf through $\omega$ and is disjoint from $S$. As before, this is possible because $S$ is in efficient position. The prong curve $c$, just like the orbit in the case above, determines a dual cycle $\gamma$ with $\Psi(\gamma) = c$ and we proceed as before.
\qedhere

\end{proof}

With these tools at hand, we can prove the main result of this section. Recall that by \Cref{rmk:adjust}, $B \subset \mr M_{\mr S}$.

\begin{proposition} \label{lem:surface-represents-class}
When $\eta \in H^1(M\cut S)$ is nonnegative, there is a surface $\mr \Sigma$ carried by $B$, and hence properly embedded in $\mr M \cut \mr S$, such that for every $1$-cycle $z$ in $\mr M_{\mr S}$, $\langle \mr\Sigma, z \rangle_{\mr M_{\mr S}} = \langle \eta, z\rangle_{M\cut S}$. In other words, $\mr \Sigma$ represents the pullback of $\eta$ to $H^1(\mr M_{\mr S})$.
\end{proposition}

\begin{proof}
Recall that we have an embedding $i \colon \Gamma \cut S \to N = M\cut S$. Since $i^*\eta$ is nonnegative on the directed cycles of $\Gamma \cut S$ by \Cref{lem:pullback},
\cite[Lemma 5.10]{LMT20} implies that $i^*\eta$ is represented by a nonnegative rational cocycle $c'$ on $\Gamma \cut S$. (The statement in \cite{LMT20} is for real coefficients, but the results are also valid for rational coefficients.) We choose a positive integer $n$ so that $c = n \cdot c'$ is an integer cocycle.

The cocycle $c$ determines a weight system on $B$ that satisfies the matching equations for each edge in the branching locus. That is, the sum of the weights on one side of an edge matches the sum of weights on the other side. This is a consequence of the fact that the edges in the branching locus of $B$ correspond to edges of $\tau^{(2)}$ that are \emph{not} traversed by $\mr S$ and that $\eta$ is a class in $H^1(M \cut S)$. Indeed, if we take a small oriented closed (nulhomotopic) loop $l$ around an edge $e$ in the branching locus of $B$, it is disjoint from $S$ and hence $\eta(l) = 0$. Since the faces on each side of $e$ are consistently cooriented, this implies that the sum of the $c$--weights on one side of $e$ equals the sum of the $c$--weights on the other.

We conclude that the weight system induced by $c$ determines a properly embedded surface that is carried by $B$ (see \cite[Section 1]{oertel1986homology}). 
We denote this properly embedded surface by $\mr \Sigma \subset \mr N$ and realize it to be transverse to the fibers of $W$.

Note that if $c'$ is positive, then $\mr \Sigma$ is fully carried by $B$; that is, it traverses every branch. We next show that $\mr \Sigma$ represents the pullback of $n \cdot \eta$ to $H^1(\mr M_{\mr S})$. For this, first observe that the equation in the proposition statement holds for any $1$-cycle in $\Gamma \cut S$ by construction, so it suffices to know that the inclusion $\Gamma \cut S \to \mr M_{\mr S}$ is surjective on $H_1$.

To this end, let $\tau_{\mr S} \subset \mr M$ be the subbranched surface of $\tau^{(2)}$ that is a union of sectors traversed by $\mr S$. Then $\mr M \cut \tau_{\mr S}$ is a manifold with a (truncated) ideal triangulation inherited from $\tau$. By construction, the dual directed graph of $\mr M \cut \tau_{\mr S}$ is exactly $\Gamma \cut S$ and so the inclusion $\Gamma \cut S \to \mr M \cut \tau_{\mr S}$ is $\pi_1$--surjective over every component of $\mr M \cut \tau_{\mr S}$. 

Finally, note that $\mr M \cut \tau_{\mr S}$ can be obtained from $\mr M_{\mr S}$ by collapsing its vertical boundary along $I$-fibers. In particular, they have homeomorphic interiors. From this, we conclude that  $\Gamma \cut S \to \mr M_{\mr S}$ is surjective on $H_1$.

To complete the proof, it only remains to apply a lemma of Thurston (\cite[Lemma 1]{thurston1986norm}) to know that $\mr \Sigma$ is a disjoint union of $n$ homologous surfaces. After replacing $\mr \Sigma$ with one of these surfaces, it represents the pullback of $\eta$ to $\mr M_{\mr S}$, as required.
\end{proof}

\begin{remark}
The proof additionally shows that when $\eta$ is positive, we can take $\mr \Sigma$ to be fully carried by $B$ and represent \emph{a multiple of} $\eta$. In particular, $\mr \Sigma$ crosses each edge of $\Gamma \cut S$. Indeed, by \cite[Lemma 5.10]{LMT20}, when $\eta$ is positive,  we may take the cocycle $c'$ to be positive on every edge, so corresponding carried surface $\mr\Sigma$ is fully carried by $B$. However, the surface representing the pullback of $\eta$ from Thurston's lemma might fail to be fully carried.
\end{remark}

\section{Transverse surfaces in $M \cut S$ and proofs of the main theorems}
\label{sec:transverse}
\label{sec:fill}

\medskip
The surface $\mr\Sigma$ we built in \Cref{sec:nonneg-class} is properly embedded in $\mr N$. In this section, we describe a procedure to add back in the components of $U \cut S$ to $\mr N$ to reproduce $N = M \cut S$ in a way so that $\mr \Sigma$ extends to a properly embedded almost transverse surface $\Sigma$ in $N$ representing $\eta$. The method for this is a generalization of \cite[Theorem 4.1]{landry2025transverse} and we refer the reader there for additional details whenever possible.

Recall from \Cref{sec:branched} that given a transverse $S$, the components of $U$ are divided into disjoint, meridian, or ladderpole tubes, depending on their intersections with $S$. In Steps $1$--$3$, we fill in the three types of components of $U \cut S$ and attempt to extend $\mr \Sigma$. However, it is not always possible to extend $\mr \Sigma$ locally while filling in these components, so after these steps we are not guaranteed to have a properly embedded surface in $N$. We will deal with the relative boundary components produced in previous steps (Steps 2 and 3) simultaneously in Steps 4.

\subsection*{Step 1}\label{step1}
We first fill in the disjoint tubes in $\mr N$. To do this, we exactly follow the same process as in the closed surface case (\cite[Section 4]{landry2025transverse}) to blow up $\phi$ inside the filled-in tubes so that $\mr \Sigma$ can be extended to a transverse surface $\mr \Sigma_1$ properly embedded in the resulting manifold. To save notations, we continue to call the blowup flow $\phi$.

\subsection*{Step 2}
Next, we fill in the components of $U\cut S$ coming from meridian tubes. Let $U_\omega$ be a meridian tube. Then $\partial U_\omega\cut \mr S$ is a collection of annuli $A_1,\cdots,A_n$. Recall from \Cref{sec:cut} that we have chosen $A_i$ to be flow parallel.  We will fill in $D^2\times I$ to each $A_i$ while modifying and extending $\mr\Sigma_1$.

The next claim is implicit in \cite{landry2019stable}, but we sketch a proof for completeness.
Recall that $\tau_\omega$ is the train track on $\partial U_\omega$ given by $\tau^{(2)}\cap\partial U_\omega$.

\begin{claim}\label{claim:meridian-bounds}
Suppose $m$ is a meridian on $\partial U_\omega$ carried by $\tau_\omega$, then $m$ bounds an embedded transverse disk in $U$.
\end{claim}

\begin{proof}[Sketch]
The tube $U_\omega$ has some transverse meridian disk $D$ and by considering the complementary bigons of $\tau_\omega$ crossed by $\partial D$, we see that $\partial D$ can be isotoped to be carried by $\tau_\omega$. This isotopy can be extended to $D$, supported on a small neighborhood of $\partial U_\omega$, so that $D$ remains transverse to the flow. Now if $m$ is other curve that is carried by $\tau_\omega$ and isotopic to $\partial D$, then this isotopy can be achieved in $\partial U_\omega$ by `pushing across' the bigons of $\tau_\omega$. As before, this isotopy can be extended to $D$ so that it remains transverse. This completes the proof.
\end{proof}

If a component of $\partial \mr \Sigma_1\cap A_i$ is a meridian, it is carried by $\tau_\omega$ because $\mr\Sigma$ is carried by $\tau^{(2)}$ by construction. By \Cref{claim:meridian-bounds}, we can cap the boundary off with a transverse disk.

If a component $\alpha$ of $\partial \mr \Sigma_1\cap A_i$ is an arc $\alpha$ with both ends on 
$\mr S^+$, then $\alpha$ together with an arc in $\partial^+ A_i$ bounds a flow saturated disk $D$ in $A_i$. We modify $\mr\Sigma_1$ by adding $D$ to $\mr \Sigma_1$ and tilting $D$ slightly to make it transverse to $\phi$. A similar operation can be done when $\alpha$ has both ends on $\mr S^-$.

Now we are left with arcs in $\partial \mr \Sigma_1\cap A_i$ with one end on $\mr S^+$ and the other end on $\mr S^-$. We adjust $\mr\Sigma_1$ such that such intersection arcs are vertical fibers on $A_i$, while keeping the interior of $\mr\Sigma_1$ transverse. As indicated in our preliminary remarks at the start of this section, we will deal with these boundary components in Step 4. We call the resulting surface after this step $\mr\Sigma_2$.

\subsection*{Step 3}
Next we fill in the components of $U\cut S$ coming from ladderpole tubes. Suppose $U_\omega$ is a ladderpole tube. By assumption, the intersection of $S$ with $\partial U_\omega$ is a collection of disjoint embedded closed curves, each carried by a ladderpole slope on $\tau_\omega$, and they cut $\partial U_\omega$ into annular components.

Take any component $C$ of $U_\omega\cut S$, and assume that $\partial C$ meets $\partial \mr\Sigma_2$. For each annulus $A\subset\partial C\cap\partial U_\omega$ intersecting $\partial\mr\Sigma_2$, the intersection $A\cap\partial\mr\Sigma_2$ is either a closed curve carried by a ladderpole, or a collection of arcs connecting the two components of $\partial A$. The two cases are illustrated in \Cref{fig:l-tubes}. We claim that fixing $C$, only one of these two possibilities can happen among all $A\subseteq\partial C\cap\partial U_\omega$.

\begin{figure}[h]
    \centering
    \labellist
    \pinlabel $C$ at 140 150
    \pinlabel $C$ at 860 150
    \endlabellist
    \includegraphics[width=4in]{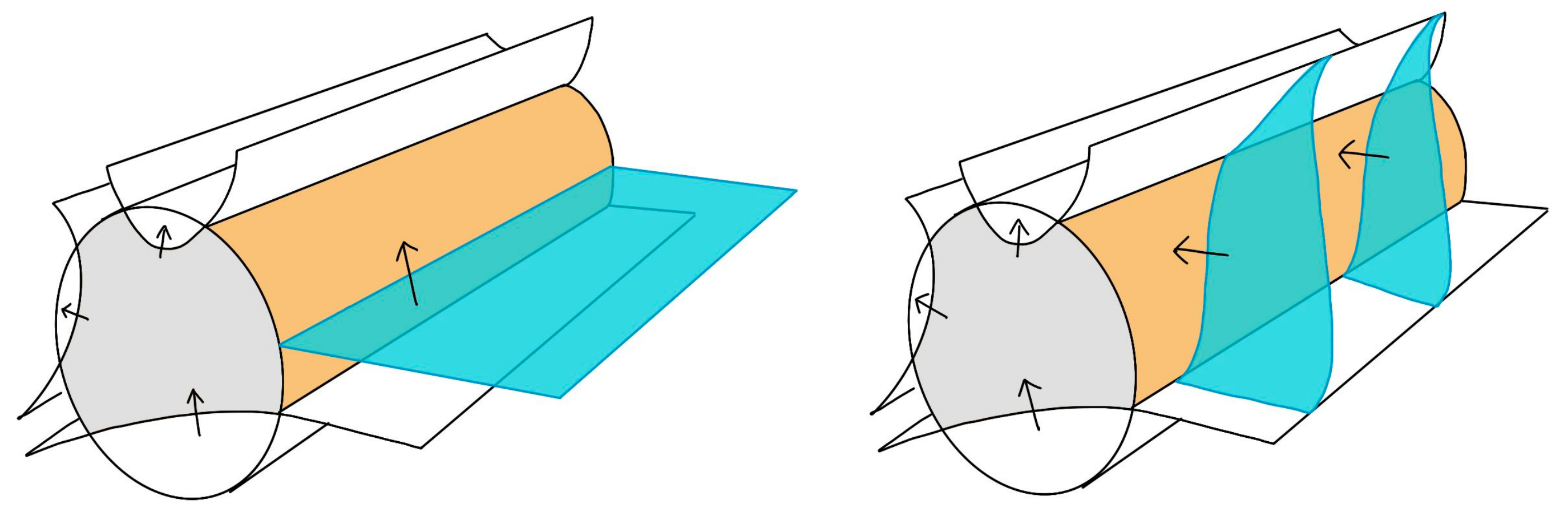}
    \caption{Two possible ways in which $\mr\Sigma_2$ (in blue) can intersect a boundary annulus $A$ (in orange) of $C$.}
    \label{fig:l-tubes}
\end{figure}

First, this is true for each individual $A$ because $\mr\Sigma_2$ is embedded. If $A\cap\partial\Sigma_2$ is a collection of arcs, then the arcs are carried by $\tau_\omega|_A$ and therefore have the same coorientation (see \Cref{fig:prongs}). In particular, we have the intersection number $\langle [\mr\Sigma],\mathrm{core}(A)\rangle_{\mr M_{\mr S}}\neq 0$.
On the other hand, if $A\cap\partial\Sigma_2$ is a collection of ladderpoles, then $\langle [\mr\Sigma],\mathrm{core}(A)\rangle_{\mr M_{\mr S}}= 0$. 
But by \Cref{lem:surface-represents-class}, we have
\[
\langle [\mr\Sigma],\mathrm{core}(A)\rangle_{\mr M_{\mr S}}= \langle\eta,\mathrm{core}(A)\rangle_N
\]
where $\langle\eta,\mathrm{core}(A)\rangle_N$ is independent of the annular component $A$ of $\partial C\cap\partial U_\omega$ because the cores of all such annuli are homotopic in $N$.
Therefore, the intersection pattern of $\mr\Sigma_2$ with any $A$ in $\partial C\cap\partial U_\omega$
is determined by $\langle\eta,\mathrm{core}(C)\rangle_N$, which only depends on $C$.

We will now fill in the components of $U_\omega\cut S$. The way we extend $\mr\Sigma_2$ across each $C$ depends on how they intersect along the boundary; the cases where they intersect in ladderpoles and in arcs are treated separately. The above analysis shows that the two different cases do not interact in $N$.

\medskip

\noindent \textit{Case 1:} As the first case, consider a component $C$ of $U_\omega\cut S$ such that $\partial\mr\Sigma_2\cap\partial C$ is a collection of ladderpoles. We resolve all such intersections by extending $\mr\Sigma_2$ across $C$. Similar to \cite{landry2025transverse}, this is done by joining the ladderpole intersections and $\partial N$
using annuli in $C$, and possibly blowing up the flow $\phi$ to preserve transversality.

\begin{figure}[h]
    \centering
    \labellist
    \pinlabel \textcolor{LimeGreen}{$A_S$} at 400 1730
    \pinlabel \textcolor{BurntOrange}{$A'$} at 1514 1620
    \pinlabel \textcolor{red}{$A''$} at 530 725
    \endlabellist
    \includegraphics[width=5in]{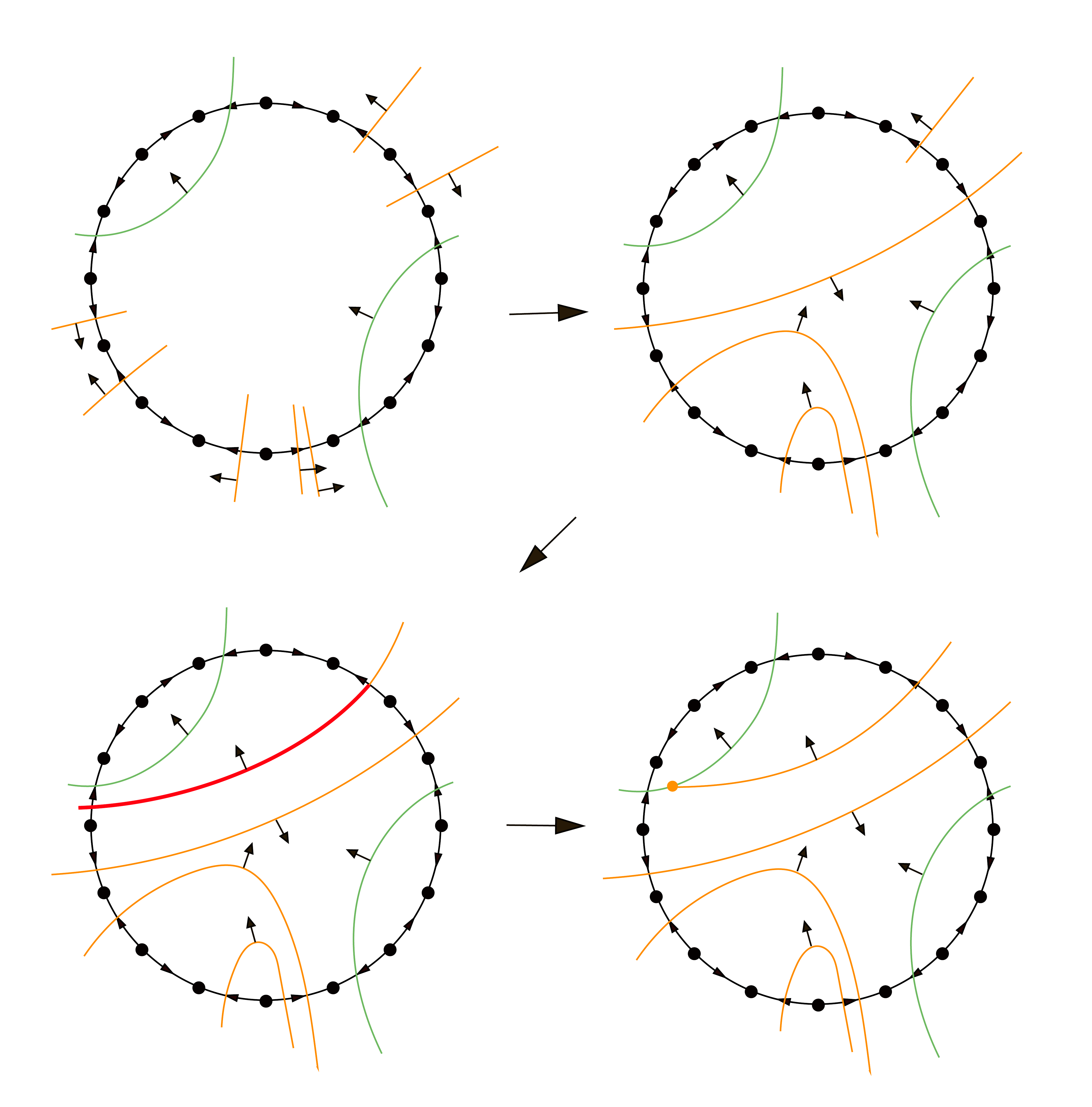}
    \caption{Steps for determining the blowup and extending $\mr\Sigma_2$}\label{fig:blow-up}
\end{figure}

We describe the blowup in greater detail (see also \Cref{fig:blow-up} for an example of the procedure). Following \cite[Section 4.1]{landry2025transverse}, pick a transverse meridian $D$ of $U_\omega$ 
 and let $f$ be the first return map. The blowup complex associated to $\omega$ intersects $D$ in a blowup tree $T$, which is preserved by $f$. We define $T^+ \subset D$ be the union of $T$ together with the stable/unstable prongs in $D$ that join the vertices of $T$ to the boundary of $D$. 
 Each vertex of $T^+$ has even degree with adjacent edges alternating between attracting and repelling under powers of $f$. The arcs in $\partial D\ssm T^+$
 are also cooriented, so that the intersection of an expanding prong of $T^+$ with $\partial D$ is a sink, and the intersection of a contracting prong of $T^+$ with $\partial D$ is a source. As in \cite[Section 4.1]{landry2025transverse}, the annuli $S\cap U_\omega$ intersect $D$ in a \emph{coherent arc system} $A_S$ on $D$; in particular, the intersection is a collection of embedded arcs that meet the edges of $T\cap D$ with compatible coorientations (see \cite[Proposition 4.3]{landry2025transverse}). In fact, the blowup tree $T$ is exactly the dual graph to the arc system $A_S$ in $D$. We now turn to build a new blowup complex in $U_\omega$ from a new coherent arc system on $D$ that sees both $S$ and $\mr\Sigma_2$. As we will see later, the associated blowup is a further blowup of $\phi$, and transverse to some extension of $\mr\Sigma_2$ in $U_\omega$.

We first describe the construction assuming $f$ fixes the prongs. The intersection $D\cap C$ is a complementary region $D_C$ of the arc system $A_S$ in $D$. Note that $\partial\mr\Sigma_2\cap \partial D_C$ is represented by a collection $P$ of compatibly cooriented points on $\partial D_C\cap\partial D$ disjoint from the sinks and sources. If there are two points in $P$ with opposite coorientations viewed from $\partial D$, we can pair them up using a compatibly cooriented arc in $D_C$. We do a sequence of such pairings, each time ensuring that the unpaired points can be connected to $A_S$ by paths away from existing pairing arcs.

After finitely many pairings, we call the set of remaining points, if any, $P'$.
Denote the resulting collection of arcs by $A'$.
Now all the points in $P'$ are in the components of $\partial D_C\ssm \partial A'$ containing $A_S$. Lastly, we can draw a coherent arc, contained in $D_C$ and disjoint from $A'$, from each point in $P'$ to a quadrant in $\partial D$ containing an endpoint of $A_S$. Call this collection of arcs $A''$. Now we have a coherent arc system $\mathcal{A}:=A_S\cup A'\cup A''$ in $D$ and so we can proceed exactly as in \cite[Section 4]{landry2025transverse}. In particular, starting with \cite[Section 4.2]{landry2025transverse} we can find a blowup tree dual to $\mathcal{A}$, which determines a blowup of the blowdown of $\phi$ at $\omega$. Moreover, since $\mathcal{A}$ contains $A_S$, we can choose the associated blowup to be a further blowup of $\phi$. Finally, by suspending the blowup tree and merging the annuli from $A''$ into $S$ (as in \cite[Section 4.3]{landry2025transverse}) we get the desired extension of $\mr \Sigma_2$ in $C$. 

\smallskip
Now we consider the case where $f$ permutes the prongs via a non-trivial rotation. There is a finite-order rotation $r$ of the meridian $D$ acting as the same permutation on the prongs as $f$, and we will use $r$ as a model of the first return map $f$. We would like to construct a cooriented arc system $\mathcal{A}$ as before, but in an $r$--invariant way. To this end, we work in the orbifold quotient $\widehat{D}=D/r$ with a single orbifold point $o$ in the center. Each annular component of $S\cap U_\omega$ gives a single arc in $\widehat{D}$, and we denote the collection of these arcs by $\widehat{A}_S$. The intersection $D_C=C\cap D$ descends to a single component $\widehat{D}_C$ of $\widehat{D}\ssm\widehat{A}_S$, and $\partial\mr\Sigma_2\cap\partial D_C$ descends to a collection $\widehat{P}$ of compatibly cooriented points on $\partial\widehat{D}_C\cap\partial\widehat{D}$. We run the same construction as before and produce a compatibly cooriented arc systems $\widehat{\mathcal{A}}=\widehat{A}_S\cup \widehat{A}' \cup \widehat{A}''$. Up to small perturbation, we can assume the arcs in $\widehat{\mathcal{A}}$ are disjoint from $o$. Then we can take the preimages of $\widehat{\mathcal{A}}$ in $D$ to get the desired $r$-invariant arc system $\mathcal{A}$, which again determines a blowup tree and a blowup of the blowdown of $\phi$. The arc system $\mathcal{A}$ can also be divided into $A_S\cup A'\cup A''$ by construction, and the partition is also $r$-invariant. To extend $\mr\Sigma_2$, we suspend $A'\cup A''$ along the flow, and merge the annuli coming from $A''$ into $S$. This completes the construction when the flow around $U_\omega$ permutes the prongs.

We run the above construction for all components $C$ of $U_\omega\cut S$ such that $\partial\mr\Sigma_2$ intersects $\partial C$ in ladderpoles, and obtain an extension of $\mr\Sigma_2$ in all such $C$ and a further blowup of $\phi$ transverse to the extended surface. Apart from Step 1, this is the only place in the construction where we need to modify the flow. To save notations, we continue to call the flow after the further blowup by $\phi$.
\medskip

Before moving to the second case, we introduce some terminology. Again let $C$ be a component of $U_\omega \cut S$. We call a component $A$ of $\partial U_\omega\cut \mr S$ \emph{thin} if both boundary components of $A$ are carried by the same ladderpole in $\tau_\omega$. Otherwise, we say $A$ is \emph{thick}. If $A$ is thin, then it is contained in the totally short boundary $\partial_\text{ts}W_I$.

\medskip

\noindent \textit{Case 2}: In the second case, suppose $C$ is a component of $U_\omega\cut S$ such that all intersections $\partial\mr\Sigma_2\cap\partial C$ are arcs. If there is only one component $A$ in $\partial C\cap\partial U_\omega$, then the two components of $\partial A$ are either both in $\mr S^+$ or both in $\mr S^-$. Every arc $\alpha\subset\partial \mr \Sigma_2 \cap A$ 
can be completed to a meridian $\mu$ carried by $\tau^{(2)}\cap A$, and $\mu$ bounds a transverse meridian disk $\Delta$ in $U_\omega$ by \Cref{claim:meridian-bounds}. Let $\Delta_0$ be the intersection $\Delta\cap C$. We extend $\mr\Sigma_2$ into $U_\omega$ by gluing $\Delta_0$ to $\mr\Sigma_2$ along $\alpha$ and do this for all $\alpha$.

If $\partial C\cap\partial U_\omega$ has more than one components, denote them by $A=A_1, A_2,\cdots, A_n$. Since $\mr \Sigma$ is constructed to be disjoint from $W_I$, we may assume $A_1,\cdots,A_m$ are thick, and $A_{m+1},\cdots, A_n$ are thin for some $m$. Suppose $\mr\Sigma_2$ intersects $A_1$ in $k$ arcs $\alpha_1,\cdots,\alpha_k$. Note that every $\alpha_i$ is carried by the train track $\tau_\omega$. Since the coorientation on the $\alpha_i$ is determined by the track $\tau_\omega$,
these arcs are all of the same coorientation determined by the flowlines,
and $\langle[\mr\Sigma],[\mathrm{core}(A_1)]\rangle_{\mr M_{\mr S}}= k$ if we orient $\mathrm{core}(A_1)$ using the flow direction. 

Let $n$ be a non-zero integer such that $[\mathrm{core}(A_1)]=n\cdot[\mathrm{core}(C)]$. Also note that the cores of all thick $A_i$ are homotopic in $M\cut S$.
By \Cref{lem:surface-represents-class}, we have 
\[
\langle[\mr\Sigma],[\mathrm{core}(A_i)]\rangle_{\mr M_{\mr S}}=
\langle \eta, \mathrm{core}(A_1) \rangle_{M \cut S}
=n\cdot \langle\eta,[\mathrm{core}(C)]\rangle_{M\cut S}
\]
for $i=1,\cdots,m$. In particular, for every thick $A_i$, $\mr\Sigma\cap A_i$ contains $|k|$ arcs with the same coorientation as $\mr\Sigma\cap A_1$ with $k$ divisible by $n$. We can cap off the arcs by attaching $|k/n|$ transverse meridian disks of $C$ to $\mr\Sigma_2$ (see \Cref{fig:meridian-extend}). Similar to the proof of \Cref{claim:meridian-bounds}, this can be done by first taking any $|k/n|$ disjoint embedded transverse meridian disks of $C$, and isotoping them near the boundary by `pushing across' the bigons of $\tau_\omega$ to match the arcs in $\mr\Sigma_2\cap A_i$.

However, since there might be some thin annuli in $\partial C$, the added meridian disks of $C$ have portions of boundary arcs in $\partial_\text{ts} W_I$ that are not yet capped off. We adjust so that the leftover boundaries on $\partial_{\text{ts}}W_I$ are finitely many fibers.

Let the resulting surface be $\mr\Sigma_3$. Also note that after this step we have already filled in all components of $U\cut S$ and get back $N$.

\begin{figure}[h]
    \centering
    \labellist
    \pinlabel $C$ at 270 280
    \endlabellist
    \includegraphics[width=4in]{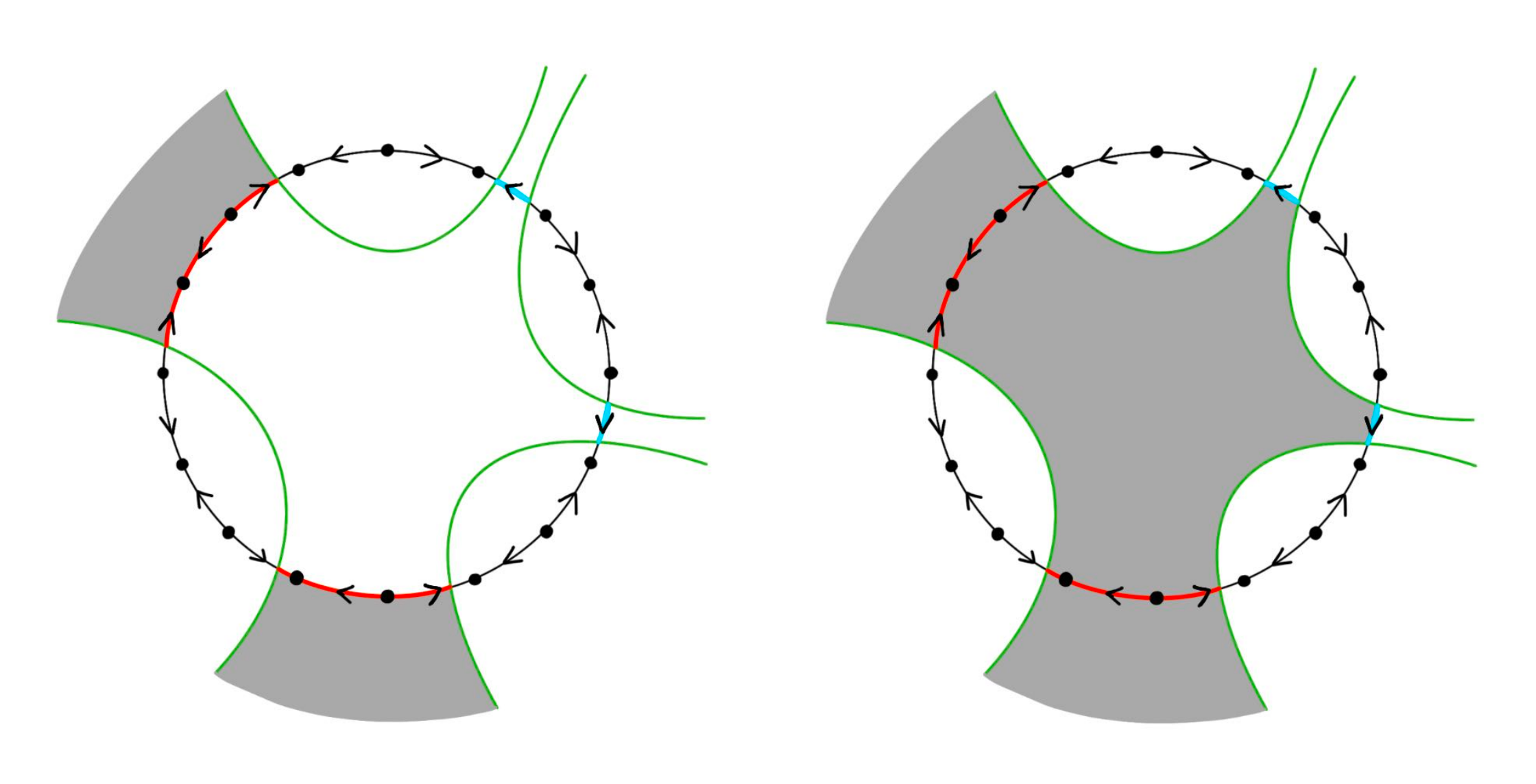}
    \caption{Extending $\mr\Sigma_2$ using a meridian disk of $C$. The result is capping off the bounday on thick annuli (red) and producing vertical boundary on thin annuli (blue).}
    \label{fig:meridian-extend}
\end{figure}

\subsection*{Step 4}
After finishing Steps 1-3, we obtained a surface $\mr\Sigma_3$ embedded in $N$
with the relative boundary $\partial\mr\Sigma_3\ssm\partial N$
coming from both Step $2$ and Step $3$. We recall that any relative boundary component coming from Step 2 is a fiber on $\partial U_\omega\cut S$ connecting $S^+$ and $S^-$, where $\partial U_\omega$ is a meridian tube.
Any relative boundary component coming from Step 3 is a fiber on a thin annulus on a ladderpole tube. In this step we will cap off the relative boundary components to obtain a properly embedded surface in $N$ representing $\eta$.

As a preparation, recall that $W_I\subset\mr N$ is an $I$-bundle fibered over $Y$. We will build a larger $I$-bundle from $W_I$ by adding in the components of meridian tubes after cutting. More precisely, if $U_\omega$ is a meridian tube and $A$ is a component of $\partial U_\omega\cut S$, recall that we made $A$ to be flow parallel by adjusting it in $N$. The effect is that each component 
of $U_\omega \cut S$ has been isotoped in $N$ to be an $I$-bundle over $D^2$ with fibers the flow lines of $\phi_N$. We attach these components to $W_I$, and the resulting manifold, denoted by $W^+_I\subset N$, is again an $I$-bundle over a surface $Y^+$, where $Y^+$ can be obtained from $Y$ by adding disks. Note that some disks are attached to existing components of $Y$ along arcs, while some are added as components.
See \Cref{fig:Y-plus}.

\begin{figure}[h]
    \centering
    \includegraphics[width=4in]{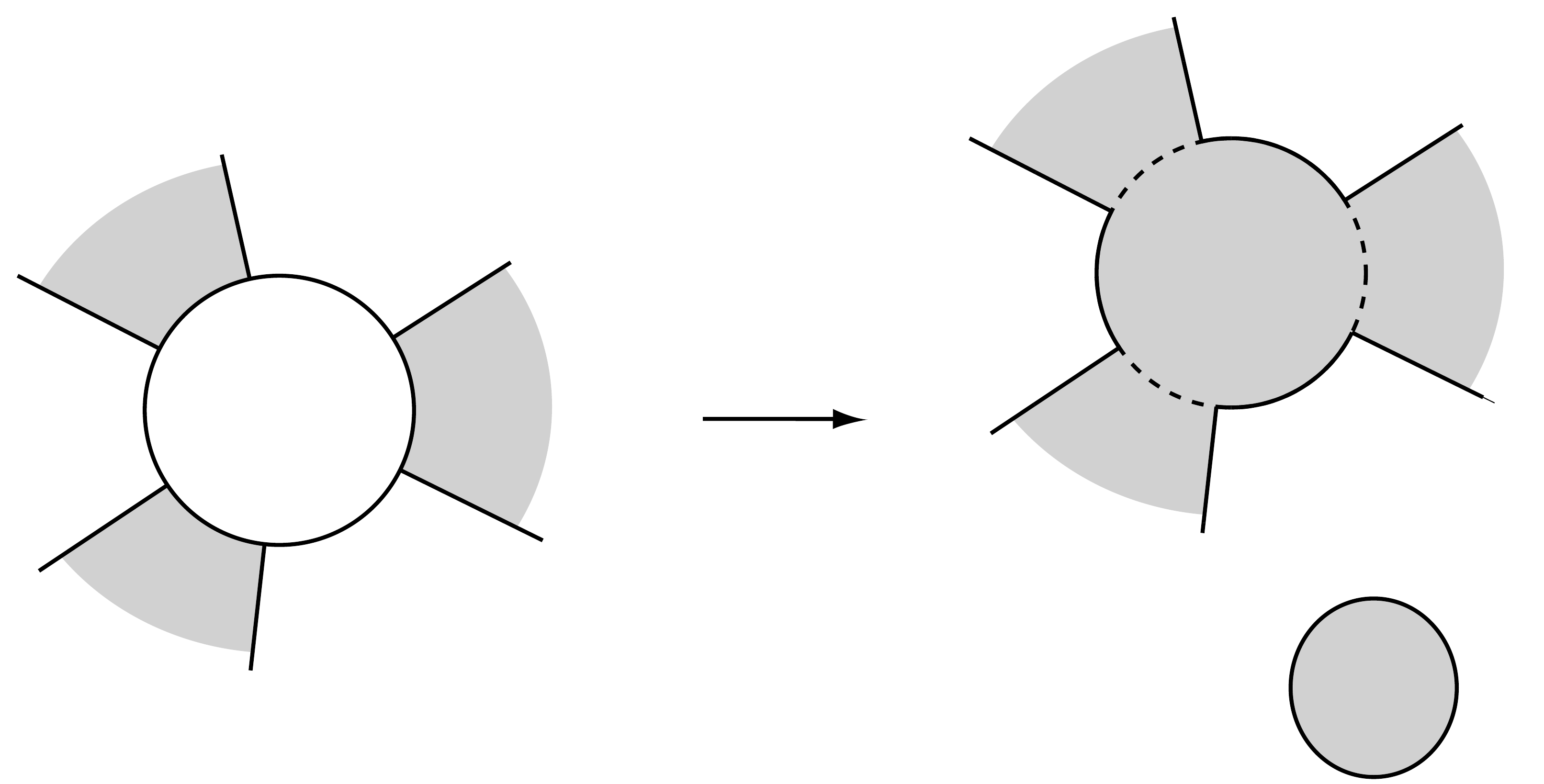}
    \caption{Adding disks to $Y$ to obtain $Y^+$}
    \label{fig:Y-plus}
\end{figure}

In this way, all the relative boundary components in $\partial\mr\Sigma_3\ssm\partial N$ are fibers on $\partial Y^+\times I$. Moreover, they inherit a natural coorientation from the transverse orientation of $\mr\Sigma_3$. The image of these relative boundary components on $\partial Y^+$ is a collection $Q$ of cooriented points. Then $Q$ gives an element $[Q]$ in $H_0(\partial Y^+)= H^1(\partial Y^+)$, where the sign of a point $q$ is positive if the coorientation of $q$ matches with the boundary orientation on $\partial Y^+$, and is negative otherwise.

After embedding $Y^+$ in $N_I$ as a cross section, we have $\partial Y^+\subseteq \mr M_{\mr S}$. By \Cref{lem:surface-represents-class} and the construction, 
$[Q]$ represents the restriction of $\eta$ on $\partial Y^+$. Therefore, we can find a cooriented arc system $\Xi$ in $Y^+$ with $\partial \Xi=Q$ and with compatible coorientations with $Q$, such that $\Xi$ represents the restriction of $\eta$ on $H^1(Y^+) \cong  H_1(Y^+,\partial Y^+)$. The preimage of $\Xi$ in $W^+_I$ is a collection $B_\Xi$ of bands fibered by flow segments. We attach $B_\Xi$ to $\mr\Sigma_3$ along the boundary fibers, and tilt $B_\Xi$ using the coorientation of $\Xi$ to make the new surface transverse. This will cap off the relative boundary $\partial\mr\Sigma_3\ssm\partial \mr N$  with annuli to obtain a properly embedded surface $\Sigma'$ transverse to the flow. By construction, the pull-backs of $[\Sigma']$ to $H^1(W_I^+)$ and $H^1(\mr M_{\mr S})$ coincide with the pull-backs of $\eta$. Finally, the Mayer-Vietoris sequence for the pair $N\cut W_I^+$ and $W_I^+$ shows that the difference between $\Sigma'$ and $\eta$ is supported on $\partial Y^+\times I$, which is a collection of flow parallel annuli. 
We can perform oriented sum to $\Sigma'$ and $\partial Y^+\times I$ to adjust the cohomology class $[\Sigma]$ to be exactly $\eta$, and further perturb the resulting surface $\Sigma$ to be transverse. This finishes Step 4.

\medskip

So far, we have constructed a properly embedded surface $\Sigma$ in $M\cut S$ representing $\eta$ and transverse to $\phi$ (after possibly replacing $\phi$ with a further blowup of itself in Case 1 of Step 3). It only remains to make sure that $\Sigma$ is essential.
By \cite[Lemma 3.6]{HT2026}, $\Sigma$ is incompressible because it is transverse to $\phi$. 
It is not necessarily the case that $\Sigma$ is boundary incompressible, but \cite[Lemma 3.7]{HT2026} states that boundary compressions can be performed while maintaining the property that $\Sigma$ is transverse to the flow. Hence, after finitely many such boundary compressions, we obtained a surface that is both essential and transverse to the flow.
\qed

\bigskip
We can now turn to the proof of \Cref{thm:class}.

\begin{proof}
Given a non-negative class $\eta$ in $H^1(M\cut S)$, \Cref{thm:properly-embedded-tst} produces a properly embedded essential surface $\Sigma$ in $M\cut S$ representing $\eta$ and a dynamic blowup of $\phi$ (which we continue to denote by $\phi$) that is transverse to both $\Sigma$ and $S$. After spinning $\Sigma$ around $S$ (see, for example, \cite[Section 3.2]{HT2026}) we obtain an infinite type surface $L$ accumulating only on $S$. The union $L\cup S$ is a depth one lamination $\mathcal{L}$ transverse to $\phi$.

To show that $\mc{L}$ can be completed to a depth one foliation when $\eta$ is positive, it suffices to show that $L$ intersects every orbit of $\phi|_S$, where $\phi|_S$ is the semi-flow in $M\cut S$ induced by $\phi$. This is because the leaf $L$ then has a well-defined first return map under $\phi$ whose mapping torus is naturally identified with $M \ssm S$. The foliations of $M \ssm S$ by fibers (parallel to $L$) plus the compact leaf $S$ determine the associated depth one foliation of $M$.

Consider any orbit $\gamma$ of $\phi$. If $\gamma$ intersects $\partial N$, then it intersects $L$ by the spinning construction. If $\gamma$ is a closed orbit contained in $N$, it pairs positively with $\eta$ by assumption, so it intersects $\Sigma$ and also $L$. If $\gamma$ is a non-closed bi-infinite orbit contained in $N$, then it is contained in $\mr M_{\mr S}$, and the sequence of tetrahedra it crosses is a bi-infinite directed path in $\Gamma \cut S$. Since we can take the cocycle defining (a multiple of) $\mr\Sigma$ to be positive on each directed edge, $\gamma$ has non-trivial intersection with $\mr\Sigma$. Therefore, $\gamma$ intersects $L$ non-trivially.
\end{proof}

\section{Applications to foliation cones}
The goal of this section is to prove \Cref{th:intro_cone}, which we recall here:

\begin{theorem}
Let $\varphi$ be a pseudo-Anosov flow on a closed atoroidal $3$-manifold $M$ and let $S$ be a closed almost transverse surface. The cone in $H^1(M\cut S)$
consisting of positive classes, when nonempty, is a foliation cone of $M\cut S$.
\end{theorem}

\begin{proof}
Set $N = M\cut S$ as before. The hypothesis that there is a positive class in $H^1(N)$, together with \Cref{thm:depth-one}, implies that there is a depth one foliation $\mc F$ of $M$, having $S$ as its compact leaves, so that $\phi$ is almost transverse to $\mc F$. Replace $\varphi$ with its minimal dynamic blowup so that it is transverse to $\mc F$, let $L$ be a depth one leaf of $\mc F$, and let $f \colon L \to L$ be its first return map. The intersection of each closed orbit $\gamma \subset N$ of $\varphi$  with $L$ is a finite collection of points that are periodic under $f$. If we replace $f$ with a power that fixes such a point $p$ (and its stable/unstable half-leaves) and let $\wt f$ be a lift of $f$ to the universal cover $\wt L$ of $L$ that fixes a point $\wt p$ in the preimage of $p$, then \cite[Proposition 8.2]{LMT_UC} states that $\wt f$ acts with multi sink-source dynamics on the hyperbolic boundary $\partial \wt L$ of $\wt L$. This means that it has at least $4$ fixed points that alternate between local attractors and repellers.

Now let $\mc C$ be the foliation cone of $N$ that contains the depth one foliation $\mc F|_N$, the restriction of $\mc F$ to $N$. By construction, it is transverse to the restricted semiflow $\varphi|_N$.
According to Theorem 7.1 and Corollary 7.2 of \cite{landry2023endperiodic}, there is an embedding $i \colon N \to M_h$ into a closed atoroidal $3$-manifold $M_h$ so that $M_h$ admits a suspension pseudo-Anosov flow $\phi_h$ with the property that each foliated class in $\mc C$ is represented by a depth one foliation that is transverse to $\varphi_h|_N$. In particular, $\varphi_h|_N$ is transverse to $\mc F_N$ and so it also determines a first return map $f_h \colon L \to L$. In the language of \cite{landry2023endperiodic}, $f_h$ is a \emph{spun pseudo-Anosov map}.

Since $f \colon L \to L$ and $f_h \colon L \to L$ are each representatives of the monodromy of the depth one foliation $\mc F|_N$, they are homotopic homeomorphisms. Moreover, for any lift $\wt f$ of $f$ we can lift a homotopy from $f$ to $f_h$ to obtained a lift $\wt f_h$ of $f_h$ with the same action on the boundary $\partial L$. (These are foundational result of Cantwell--Conlon from \cite{CaCo13}.) In particular, if we choose a lift $\wt f$ to fix a point $\wt p$ (and its stable/unstable half-leaves) as above, the associated lift $\wt f_h$ of $\wt f$ will also act with multi sink-source dynamics on $\partial L$. But then by Proposition 4.15 of \cite{landry2023endperiodic}, $\wt f_h$ also has a unique fixed point $\wt q$ in $\wt L$. Translating this into a statement about closed orbits of the associated flows, we have that every closed orbit of $\varphi_N$ is homotopic in $N$ to a closed orbit of $\varphi_h|_N$. 

Now let $\eta$ be an arbitrary foliated class in $\mc C$. As above, it is represented by a depth one foliation $\mc G$ of $N$ that is transverse to $\varphi_h|_N$. In particular, it is positive on the closed orbits of $\varphi_h|_N$. But since each closed orbit of $\varphi_N$ is homotopic in $N$ to a closed orbit of $\varphi_h|_N$, $\eta$ is also positive on the closed orbits of $\varphi_N$. Hence, $\eta$ is a positive class (with respect to $\phi_N$). Since positivity is an open condition and foliated classes are dense in $\mc C$, the proof is complete.
\end{proof}

We also characterize when the cone of positive classes is nonempty.

\begin{proposition}\label{prop:nonempt}
The cone of positive classes is nonempty if and only if there is a foliated nonnegative class.
\end{proposition}

\begin{proof}
Let $\eta$ be a non-negative class in $H^1(N)$. \Cref{thm:transverse-surface} gives a transverse depth one lamination $\mathcal{L}$ with $\mathcal{L}^0 = \partial N$ representing $\eta$ and almost transverse to $\phi$. A depth one leaf of $\mathcal{L}$ is obtained from spinning a properly embedded incompressible surface $\Sigma$ representing $\eta$ in $N$. Note that by the spinning construction, the choice of an incompressible $\Sigma$ is not unique. For example, we can enlarge $\Sigma$ by peeling off layers from $\partial N$. This flexibility is used later in the proof.

If $\eta$ is foliated, then there is a depth one foliation $\mathcal{F}$ representing $\eta$. Similarly we get a (non-unique) properly embedded incompressible surface $\Xi\subseteq N$ such that a depth one leaf $F$ is obtained from $\Xi$ by spinning. In the interior of $N$, the foliation $\mathcal{F}$ is a fibration over $S^1$ with fiber $F$, and there is a suspension flow $\psi$ transverse to $\mathcal{F}$ and $\partial N$. The surface $\Xi$ can be obtained from truncating $F$ by a spiraling neighborhood of $\partial N$, and push the boundary of the remaining subsurface $C\subseteq F$ to $\partial N$ along flowlines of $\psi$. The subsurface $C$ is also called a core of the monodromy. Let $\mathcal{U}^+$ (resp. $\mathcal{U}^-$) be the subset of $F$ consisting of points that limit to $\partial ^+N$ (resp. $\partial^-N$) under $\psi$ in forward (resp. backward) time.

Let $\wt{N}$ be the $\Z$-cover of $N$ dual to $\eta$. Then $\wt{N}$ is homeomorphic to
\[
\mathcal{U}^-\times\{-\infty\}\,\,\cup\,\, F\times\R \,\,\cup\,\, \mathcal{U}^+\times\{+\infty\}
\]
with the lifted suspension flow $\wt{\psi}$ in the $\R$-direction (see, for example, \cite[Section 3]{field2023end}). Since both $\Xi$ and $\Sigma$ represent $\eta$, they lift to properly embedded surfaces of $\wt N$ that we denote by $\wt{\Sigma}$ and $\wt{\Xi}$, respectively.
By the construction of $\Xi$, the interior of $\wt{\Xi}$ is contained in $C\times\R$, and under the projection $p \colon \wt{N}\to F$, $\wt{\Xi}$ maps to $C$ homeomorphically. Up to enlarging $\Sigma$ and $\Xi$ (and hence also $C$) as indicated previously, we may assume that $\partial\wt{\Sigma}=\partial\wt{\Xi}$ (see for example the discussion in \cite{landry2023endperiodic} preceding Lemma 3.7) and that $\wt{\Sigma}$ is contained in $\overline{C\times\R}$.
By identifying $\Sigma$ with $\wt\Sigma$ and composing the inclusion $\wt{\Sigma}\hookrightarrow\overline{C\times\R}$ with the projection $p$, we get a map $h\colon\Sigma\to C$ that restricts to a homeomorphism on boundary. The induced map $h_*$ on $\pi_1$ is injective because $\Sigma$ is incompressible. It is also surjective because the condition on $h|_{\partial\Sigma}$ forbids it to lift to a map to any non-trivial cover of $\Xi$. Therefore, $h$ is homotopic to a homeomorphism. Note $\wt\Sigma$ separates $C\times\R$ into two components, and the inclusion into each component induces an isomorphism on $\pi_1$ by Van Kampen's theorem. By \cite[Theorem 10.2]{Hempel3manifolds}, $\wt\Xi$ is isotopic to $\wt\Sigma$ rel boundary. 
In particular, the leaf $L$ obtained by spinning $\Sigma$ can also be completed into a depth one foliation in $N$. Therefore, $K=\mathrm{int}(N)\cut L$ is homeomorphic to $L\times [0,1]$. Recall that $\Sigma$ and $S$ are both transverse to some almost pseudo-Anosov flow $\phi$ in $M$, so $L$ can also be made transverse to $\phi$ as well.

We claim that $L$ has a well-defined first return map under the semi-flow $\phi|_S$. In other words, every forward flowline of $\phi|_S$ intersects $L$. This implies that the original class $\eta$ is positive, proving the proposition.

To see the claim, consider the universal cover $\wt{M}$ of $M$ and let $\wt{K}$ be a connected component of the preimage of $K$. To be precise, there is a map $\iota\colon K\cong L\times[0,1]\to N$ mapping $L\times\{0\}$ and $L\times\{1\}$ to $L\subset N$, and we consider a lift $\wt{\iota}\colon \wt{L}\times[0,1]\to\wt{M}$ with image $\wt{K}$. The boundary of $\wt{K}$ are two lifts $\wt{L_1}$ and $\wt{L_2}$ of $L$ in $\wt{M}$, corresponding to $\iota(\wt L\times\{0,1\})$. There exist flowlines of the lifted flow from $\wt{L_1}$ and $\wt{L_2}$ because the distance between them can get arbitrarily small in a spiraling neighborhood. In other words, the projections of $\wt{L_1}$ and $\wt{L_2}$ to the orbit space $\orb$ are not disjoint. The proof of \cite[Lemma 3.2]{HT2026} then shows that for any annulus $A$ in $K$ going from $L\times\{0\}$ to $L\times\{1\}$, the image $\iota(A)$ can be homotoped rel $L$ to be flow parallel, i.e. foliated by flowlines going from one boundary to the other. Since $K$ is a product, one can find a system of flow parallel annuli going from $L$ to itself, such that the boundary of the annuli cut $L$ into disks. The flow in the complement of these annuli has to be trivial, so every flowline starting from a point in $L$ must hit $L$ in the future. This completes the proof of the claim.\qedhere

\end{proof}

\bibliographystyle{alpha}
\bibliography{f&p.bib}

\end{document}